\documentclass[11pt]{amsart}
\usepackage{amsmath}
\usepackage{amssymb}
\usepackage{amsthm}
\usepackage{bbm} 
\DeclareFontFamily{OMX}{MnSymbolE}{}
\DeclareFontShape{OMX}{MnSymbolE}{m}{n}{
   <-6>  MnSymbolE5
   <6-7> MnSymbolE6
   <7-8> MnSymbolE7
   <8-9> MnSymbolE8
   <9-10> MnSymbolE9
   <10-12> MnSymbolE10
   <12->   MnSymbolE12
}{}
\DeclareSymbolFont{MnSyE}{OMX}{MnSymbolE}{m}{n}
\DeclareMathSymbol{\bigcupdot}{\mathop}{MnSyE}{16}
\usepackage[noend]{algorithm2e}
\usepackage[colorlinks,linkcolor=blue]{hyperref}
\usepackage{url}
\usepackage[margin=1in]{geometry}
\usepackage{xcolor}
\usepackage{comment}

\usepackage{placeins}

\clearpage  

\usepackage[color=green!40]{todonotes}

\usepackage{tikz}
\usetikzlibrary{graphs, positioning, calc, fit, shapes.geometric, decorations.markings, bending}
\tikzstyle{ball} = [circle,shading=ball, ball color=black,
    minimum size=1mm,inner sep=1.3pt]

\usepackage{enumitem}

\usepackage{needspace}  

\newtheorem{thm}{Theorem}[section]
\newtheorem{lemma}[thm]{Lemma}

\newtheorem{cor}[thm]{Corollary}
\newtheorem{prop}[thm]{Proposition}

\newtheorem{theorem}{Theorem}

\newtheorem{Definition}[thm]{Definition}
\newenvironment{defn}
  {\begin{Definition}\rm}{\end{Definition}}

\newtheorem{Example}[thm]{Example}
\newenvironment{example}
  {\begin{Example}\rm}{\end{Example}}

\newtheorem{Remark}[thm]{Remark}
\newenvironment{remark}
  {\begin{Remark}\rm}{\end{Remark}}

\newtheorem{Question}[thm]{Question}
\newenvironment{question}
  {\begin{Question}\rm}{\end{Question}}

\newtheorem{Conjecture}[thm]{Conjecture}
\newenvironment{conj}
  {\begin{Conjecture}\rm}{\end{Conjecture}}

\newtheorem{Problem}[thm]{Problem}  
\newenvironment{prob}
  {\begin{Problem}\rm}{\end{Problem}}

\DeclareMathOperator{\im}{im}

\DeclareMathOperator{\lk}{\ell k}

\DeclareMathOperator{\cir}{Cir}
\DeclareMathOperator{\Ind}{Ind}

\newcommand{\N}{\mathbb{N}}
\newcommand{\cs}{\mathcal{C}}
\newcommand{\csd}{\mathcal{D}}
\newcommand{\rk}{r}

\title[Cycle systems, coparking functions, and $h$-vectors of matroids]{Cycle systems, coparking functions, \\ and $h$-vectors of matroids}

\author{Scott Corry}
\address{Lawrence University}
\email{corrys@lawrence.edu}

\author{Anton Dochtermann}
\address{Texas State University}
\email{dochtermann@txstate.edu}

\author{Sol\'is McClain}
\email{solmcclain@gmail.com}

\author{David Perkinson}
\address{Reed College}
\email{davidp@reed.edu}

\author{Lixing Yi}
\address{University of California, San Diego}
\email{liyi@ucsd.edu}


\begin{document}

\begin{abstract}
The $h$-vector of a matroid $M$ is an important invariant related to the independence complex of $M$ and can also be recovered from an evaluation of its Tutte polynomial. A well-known conjecture of Stanley posits that the $h$-vector of a matroid is a pure O-sequence, meaning that it can be obtained by counting faces of a pure multicomplex. Merino has established Stanley's conjecture for the case of cographic matroids via chip-firing on graphs and the concept of a $G$-parking function. Inspired by these constructions, we introduce the notion of a cycle system for a matroid $M$---a family of cycles (unions of circuits) of $M$ with overlap properties that mimic cut-sets in a graph. A choice of cycle system on $M$ defines a collection of integer sequences that we call coparking functions.  We show that for any cycle system on $M$, the set of coparking functions is in bijection with the set of bases of $M$. We show that maximal coparking functions all have the same degree, and that cycle systems behave well under deletion and contraction. This leads to a proof of Stanley’s conjecture for the case of matroids that admit cycle systems, which include, for instance, graphic matroids of cones as well as $K_{3,3}$-free graphs.
\end{abstract}

\maketitle

\section{Introduction} 

The theory of \emph{chip-firing} on a graph $G$ has led to surprising applications in algebra, geometry, and combinatorics. In the version of chip-firing most relevant to us, one fixes a sink vertex $q$ of $G$ and considers nonnegative configurations of chips $\vec{a} \in {\mathbb N}^n$ on the nonsink vertices $\widetilde{V}=\{v_1, \dots, v_n\}$. When a vertex $v_i$ \emph{fires}, one chip is sent along each edge incident to $v_i$, decreasing the number of chips $a_i$ at $v_i$ while increasing the number of chips at the neighboring vertices. Firing $v_i$ is \emph{legal} for $\vec{a}$ when $a_i\ge\deg_G(v_i)$, so that the resulting configuration is still nonnegative. More generally, firing a subset $S\subseteq\widetilde{V}$ of vertices is legal for $\vec{a}$ if, after firing all vertices in $S$, the resulting configuration is nonnegative. A configuration $\vec{a}$ is \emph{superstable} if no nonempty set of $\widetilde{V}$ may be legally fired. Superstable configurations can be seen to coincide with the set of \emph{$G$-parking functions}, integer sequences that can be defined in terms of edge cuts of $G$. These sequences in turn generalize the well-studied \emph{parking functions} that originated in queuing theory and computer science, corresponding to the special case of complete graphs. We refer to \cite{CorryPerkinson18} for more regarding the theory of chip-firing.

The $G$-parking functions on a graph $G$ are also closely related to the \emph{Tutte polynomial} $T_G(x,y)$.  To recall this connection, let $g = |E(G)| - |V(G)| + 1$ denote the \emph{corank} of the (connected) graph $G$, and for all $i \geq 0$ let $c_i$ denote the number of $G$-parking functions of degree $g-i$ (where the \emph{degree} of a sequence is simply the sum of its entries). In \cite{MerinoTutte}, Merino showed that the generating function for $G$-parking functions is, in fact, an evaluation of the Tutte polynomial along the line $x=1$, so that
\[\sum_{i=0}^g c_i y^i = T_G(1,y).\]

A \emph{matroid} $M$ is a combinatorial object that generalizes and unifies several notions of independence that arise in disparate settings. In particular, a finite graph $G$ defines a \emph{graphic matroid} $M(G)$, whose independent sets are given by collections of edges that do not contain a cycle. 

Of particular interest in this setting is the \emph{$h$-vector} of $M$, a nonnegative sequence of integers that relates to the $f$-vector of the independence complex of $M$, and which can also be recovered as an evaluation of the Tutte polynomial $T_M(x,y)$ along the line $y=1$.

Understanding the structure of $h$-vectors of arbitrary matroids has been an area of active research in recent years. For instance, in \cite{BST}, Berget, Spink, and Tseng showed that the entries form a \emph{log-concave} sequence, generalizing  an earlier result of Huh \cite{Huh}, who established the result for matroids realizable over a field of characteristic zero (which includes graphic matroids). An open question regarding the $h$-vector of a matroid is the following conjecture of Stanley. 
In what follows, a sequence $(h_0, h_1, \dots)$ of integers is an \emph{O-sequence} if there exists an order ideal $M$ of monomials containing exactly $h_i$ monomials of degree $i$. The sequence is \emph{pure} if the maximal elements of $M$ (under divisibility) all have the same degree. In other language, a pure O-sequence counts the elements in a pure multicomplex.

\begin{conj}\cite{Stanleybook}\label{conj:Stanley}
The $h$-vector of a matroid is a pure O-sequence.
\end{conj}

Conjecture \ref{conj:Stanley} has led to a substantial body of work, and for instance has been established for lattice-path matroids by Schweig in \cite{Schweig}, cotransversal matroids by Oh in \cite{Oh}, paving matroids by Merino, Noble, Ram\'irez-Ib\'a\~nez, and Villarroel-Flores in \cite{MNRV}, positroids by He, Lai, and Oh in \cite{HLO}, internally perfect matroids by Dall in \cite{Dall}, rank 3 matroids by H\'a, Stokes, and Zanello in \cite{HSZ}, rank 3 and corank 2 matroids by DeLoera, Kemper, and Klee in \cite{DKK}, rank 4 matroids by Klee and Samper in \cite{KleeSamper}, rank $d$ matroids with $h_d \leq 5$ by Constantinescu, Kahle, and Varbaro in \cite{CKV}.
The conjecture is mostly open for graphic matroids, although it has been established for `coned' graphs by Kook in \cite{Kook} and for `bi-coned' graphs by Preston et al.~in \cite{CDHMOT}.

Another special case of Stanley's conjecture can be recovered from Merino's work. He shows that the numbers $c_i$ described above are the entries of the $h$-vector of $M(G)^*$, the matroid that is \emph{dual} to the graphic matroid $M(G)$. One can also show that for any graph $G$, the \emph{maximal} $G$-parking functions all have the same degree \cite{Merino2001}, and so generate a pure multicomplex. This establishes Conjecture \ref{conj:Stanley} for the case of \emph{cographic} matroids.

\subsection{Our contributions}
What about other matroids? In this paper we seek to generalize the results and constructions discussed above to more general matroidal settings, with a view towards Stanley's conjecture.
Our initial observation is that the conditions defining a $G$-parking function can be phrased in terms of the cut-sets (cocircuits) determined by subsets of the vertex set. This is a typical strategy when one seeks to generalize graph theoretic results to the matroidal setting: we `replace' the set of vertices of a graph with the cocircuits determined by each vertex. 
In addition, the bijection between $G$-parking functions and spanning trees can be seen to rely on the following overlap property of cocircuits: for any subset $S \subset V(G)$ of vertices, the set of edges in the cut determined by $S$ is precisely the collection of elements that appear \emph{exactly once} in the multiset of edges produced by the union of the edge-cuts corresponding to the individual vertices of $S$.  

For a graph $G$, it is sometimes possible to find a collection of \emph{circuits} of $G$ that also satisfies this overlap property. In particular, this will always be true for planar graphs, where the circuits can be identified with the cocircuits of the planar dual $G^*$.
We illustrate this with an example that we will revisit later, in Example~\ref{example: intro}. Let $G$ be the (multi)graph on 4 vertices with 7 edges shown in Figure~\ref{fig:intro}, where $q$ is the designated sink vertex.

Firing the vertex $w_1$ sends one chip along each of the edges $3,4,7$, hence will be legal for any configuration $(a_1,a_2,a_3)$ with $a_1\ge 3$. The net effect of firing the set $S=\{w_1,w_2\}$ will be to send one chip along each of the edges $2,4,6,7$, and so will be legal for any configuration with $a_1,a_2\ge 2$. To view these facts via cocircuits, let $C_1=\{3,4,7\}$ and $C_2=\{2,3,6\}$ be the edge-cuts determined by $w_1$ and $w_2$ respectively. Then the set of edges appearing exactly once in the multiset-union of $C_1$ and $C_2$ is $\{2,4,6,7\}$, precisely the edge-cut determined by the vertex set $S=\{w_1,w_2\}$. 

Since $G$ is planar, its dual matroid $M(G)^*$ is defined by the dual graph $G^*$ shown in Figure~\ref{fig:intro}.

\begin{figure}[ht]
  \centering

\begin{tikzpicture}[scale=1]
    \tikzset{mystyle/.style={circle,minimum size=4mm,inner sep=0pt,fill=gray!20,thin,solid}}

    \coordinate (c0) at (0,0);
    \coordinate (c1) at (2,0.5);
    \coordinate (c2) at (1,1.732);
    \coordinate (c3) at (-1,1.732);
    \coordinate (c4) at (-2,0.5);

    \coordinate (012) at ($0.334*(c0) + 0.333*(c1) + 0.333*(c2)$);
    \coordinate (023) at ($0.334*(c0) + 0.333*(c2) + 0.333*(c3)$);
    \coordinate (034) at ($0.334*(c0) + 0.333*(c3) + 0.333*(c4)$);

    \node[ball] (q) at (0,2.9) {};
    \node[ball,label={[xshift=1ex]below:~$w_{1}$}] (w1) at (034) {};
    \node[ball,label={below:~$w_{2}$}] (w2) at (023) {};
    \node[ball,label={[xshift=-1ex]below:~$w_{3}$}] (w3) at (012) {};
    \draw[thick] (w3)--(w2);
    \draw[thick] (w2)--(w1);
    \draw[thick] (q) to[out=0,in=30,looseness=1.5] (w3);
    \draw[thick] (q) to[out=0, in=90] (3,0) to[out=270, in=-40] (w3);
    \draw[thick] (q) to[out=180,in=150,looseness=1.5] (w1);
    \draw[thick] (q) to[out=180, in=90] (-3,0) to[out=270, in=220] (w1);
    \draw[thick] (q)--(w2);

    \node[mystyle] at (2.4,1.8) {\small $1$};
    \node[mystyle] at ($(c0)!0.55!(c2)$) {\small $2$};
    \node[mystyle] at ($(c0)!0.55!(c3)$) {\small $3$};
    \node[mystyle] at (-2.4,1.8) {\small $4$};
    \node[mystyle] at (0,2) {\small $6$};
    \node[mystyle] at (1.6,1.8) {\small $5$};
    \node[mystyle] at (-1.6,1.8) {\small $7$};

    \node at (0,-1) {$G$};
 
    \begin{scope}[xshift=8cm]
      \coordinate (c0) at (0,0);
      \coordinate (c1) at (2,0.5);
      \coordinate (c2) at (1,1.732);
      \coordinate (c3) at (-1,1.732);
      \coordinate (c4) at (-2,0.5);
      \node[ball,label={below:$v_0$}] (0) at (0,0) {};
      \node[ball,label={right:$v_1$}] (1) at (2,0.5) {};
      \node[ball,label={above:$v_2$}] (2) at (1,1.732) {};
      \node[ball,label={above:$v_3$}] (3) at (-1,1.732) {};
      \node[ball,label={left:$v_4$}] (4) at (-2,0.5) {};

      \coordinate (012) at ($0.334*(c0) + 0.333*(c1) + 0.333*(c2)$);
      \coordinate (023) at ($0.334*(c0) + 0.333*(c2) + 0.333*(c3)$);
      \coordinate (034) at ($0.334*(c0) + 0.333*(c3) + 0.333*(c4)$);

      \begin{scope}[opacity=0.2]
	  \node[ball] (q) at (0,2.9) {};
	  \node[ball] (w1) at (034) {};
	  \node[ball] (w2) at (023) {};
	  \node[ball] (w3) at (012) {};
	  \draw[thick] (w3)--(w2);
	  \draw[thick] (w2)--(w1);
	  \draw[thick] (q) to[out=0,in=30,looseness=1.5] (w3);
	  \draw[thick] (q) to[out=0, in=90] (3,0) to[out=270, in=-40] (w3);
	  \draw[thick] (q) to[out=180,in=150,looseness=1.5] (w1);
	  \draw[thick] (q) to[out=180, in=90] (-3,0) to[out=270, in=220] (w1);
	  \draw[thick] (q)--(w2);
      \end{scope}

      \draw[thick] (0)--(1) node[mystyle,pos=0.75] {\small $1$};
      \draw[thick] (0)--(2) node[mystyle,pos=0.55] {\small $2$};
      \draw[thick] (0)--(3) node[mystyle,pos=0.55] {\small $3$};
      \draw[thick] (0)--(4) node[mystyle,pos=0.75] {\small $4$};
      \draw[thick] (1)--(2) node[mystyle,pos=0.5] {\small $5$};
      \draw[thick] (2)--(3) node[mystyle,pos=0.5] {\small $6$};
      \draw[thick] (3)--(4) node[mystyle,pos=0.5] {\small $7$};

      \node at (0,-1) {$G^*$};
    \end{scope}
  \end{tikzpicture}
  
  \caption{A graph $G$ with its dual $G^{*}$.}
  \label{fig:intro}
\end{figure}
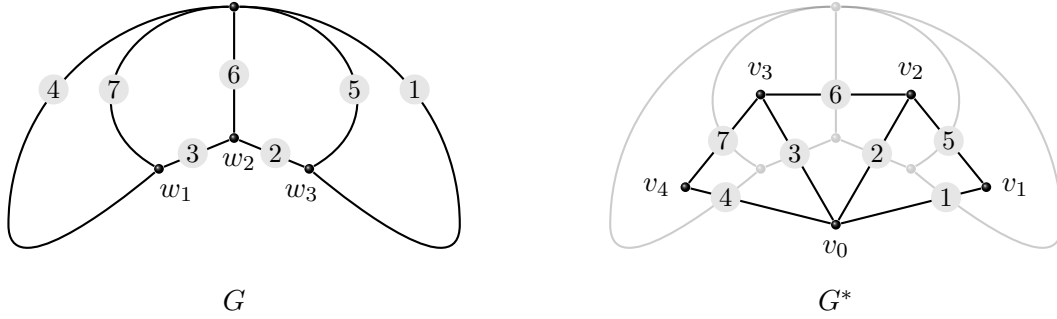

\noindent
The cocircuits defined by the nonsink vertices of the original graph $G$ correspond to the bounded faces of $G^*$; define $C_3=\{1,2,5\}$ to complement the definitions of $C_1$ and $C_2$ above. Then the $C_i$ provide a set of $g(G^*)=3$ circuits with the special property that the set of edges appearing exactly once in any nonempty multiset-union of the $C_i$ is dependent (i.e., contains a circuit). Moreover, $G$-parking functions on the original graph may instead be viewed as functions defined on the circuits $C_i$, subject to inequalities related to this overlap property. This leads to our introduction of \emph{coparking functions} in Definition~\ref{defn:coparking}.

Motivated by these observations, we define the \emph{unique union operator} and use it to introduce the notion of a \emph{cycle system} for a matroid $M$ (see Definition \ref{defn:cyclesystem}).
Here we consider collections of cycles of a matroid $M$ with the desired overlap property, where a \emph{cycle} is by definition an ordinary union of circuits of $M$.
As indicated in the previous example, the definition mimics the structure of cut-sets in a graph. Indeed, if $G$ is a graph with $n+1$ vertices, the collection of edge-cuts defined by any $n$ vertices defines a cycle system for the dual matroid $M(G)^*$. 

Our first collection of results investigates the structure of cycle systems.  Unfortunately, not all matroids admit cycle systems (the graphic matroid of the complete bipartite graph $K_{3,3}$ is one such example, see below).  Our positive results for cycle systems can be summarized as follows.

\begin{theorem}[Proposition \ref{prop:del-con cs}, Theorem \ref{thm:circuits}, Theorem \ref{thm:components}, Proposition \ref{prop:circuitspace}]
Suppose $M$ is a matroid with cycle system $\cs = \{C_1, C_2, \dots, C_g\}$.

\begin{itemize}
    \item
    For some choice of $e \in M$, the deletion $M \backslash e$ and contraction $M/e$ both admit cycle systems.
    \item If $M$ is connected, then each $C_i$ is a circuit of $M$.
    \item If $M$ is not connected, then $\cs$ can be decomposed into a sum of circuit systems on the underlying components (in the appropriate sense).
    \item If $M$ is connected and binary, then $\cs$ is a basis for the circuit space of $M$.
\end{itemize}
\end{theorem}

In Proposition \ref{prop:two-sum} we show that if $M$ and $N$ are matroids with cycle systems, then (under certain conditions) the $2$-sum $M \oplus_2 N$ also admits a cycle system. Using this we identify a large class of matroids that admit cycle systems, including all graphic matroids of graphs that are $K_{3,3}$-free (see Corollary \ref{cor:K33-free}).

We next turn to the analogue of parking functions for cycle systems. For $M$ a matroid with cycle system $\cs$, we construct a collection of integer sequences $(a_1, a_2, \dots, a_g)$ determined by overlap properties among subsets of $\cs$. Inspired by the constructions for dual graphs discussed above, we call these \emph{coparking functions} (see Definition \ref{defn:coparking}). We let~$P^{*} = P^{*}(\cs)$ denote the set of all coparking functions. The set $P^{*}$ naturally forms a poset, so one can think of $P^{*}$ as a multicomplex.

As in \cite{MerinoTutte}, we study how coparking functions behave with respect to deletion and contraction (see Proposition \ref{prop:del-con coparking}). This allows us to relate the degree vector of coparking functions to the Tutte polynomial of $M$ and in turn establish Stanley's conjecture for matroids that admit cycle systems.  Our results in this context can be informally stated as follows.

\begin{theorem}[Algorithm \ref{alg:coparking}, Proposition \ref{prop:pure}, Theorem \ref{thm:main}]
Suppose $M$ is a matroid with cycle system~$\cs$.
\begin{itemize}

\item A version of Dhar's burning algorithm can check if a given integer sequence is a coparking function with respect to $\cs$.

\item The poset of coparking functions
  $P^*(\cs)$ is a pure multicomplex.
\item 
The  degree vector of~$P^{*}$ is the~$h$-vector of~$M$, i.e., $\deg(P^{*})=h(M)$.

\end{itemize}
\end{theorem}

In Section~\ref{sec:bijections} we exhibit a bijection between the coparking functions and bases of a matroid with a cycle system in terms of \emph{deletion/contraction trees}.  This bijection, together with our resolution of Stanley's conjecture for this class of matroids, are our primary contributions.
\medskip

\noindent\textbf{Main Results.} (Algorithms \ref{alg:basis-to-coparking} and \ref{alg:coparking-to-basis}, Corollary \ref{cor:Stanley})
Suppose $M$ is a matroid with cycle system $\cs$.

\begin{itemize}
\item 
There exists an explicit bijection between the set of coparking functions defined by $\cs$ and the set of bases of $M$.

\item 
Conjecture \ref{conj:Stanley} holds for $M$.

\end{itemize}

\subsection{Related work.}

 The unique union property also plays a role in work of Dong \cite{Dong} in his definition of \emph{$B$-parking functions}; we briefly recall the relevant ideas here. A matching $M$ of a graph $G$ is \emph{uniquely restricted} if it is the only perfect matching in the subgraph of $G$ induced by the vertices of $M$. In \cite{Dong}, it is shown that for any graph $G$ with vertex set $V$, there is a bijection between the set of spanning trees of $G$ and the set of uniquely restricted matchings of size $|V|-1$ in $S(G) \backslash x_0$, where $S(G)$ is the bipartite graph obtained from $G$ by subdividing each edge. Motivated by this, Dong defines a notion of \emph{$B$-parking function} for an arbitrary bipartite graph $B$ with bipartition $(X,Y)$.  He establishes a bijection between the set of uniquely restricted matchings of $B$ and the set of $B$-parking functions.

 The definition of a $B$-parking function involves considering sets of degree-1 vertices in certain subgraphs of $B$, which mimics our constructions in some ways. In particular, for a graph $G$ with sink vertex $q$, for each $v \in V \backslash q$ let $E_v$ denote the set of edges incident to $v$. Define $B$ to be the bipartite graph whose vertex set consists of $V \backslash q \sqcup \{E_v\}_{v \in V \backslash q}$, with edges determined by incidence. We then recover our coparking functions on the cographic matroid as the $B$-parking functions on~$B$.

\subsection{Organization.} The rest of the paper is organized as follows. Section~\ref{sec: prelim} provides background on matroids, Tutte polynomials, $h$-vectors, multicomplexes, and O-sequences. In section~\ref{sec: cycle systems}, we introduce the unique union operator, define cycle systems, and prove some initial results about them, including behavior under deletion, contraction, and 2-sum. In section~\ref{sec: coparking}, we introduce coparking functions, show that they form a pure multicomplex, and use a deletion/contraction argument to prove our main result, Theorem \ref{thm:main}: the degree vector of the poset of coparking functions is the $h$-vector for the underlying matroid $M$. This establishes Stanley's Conjecture~\ref{conj:Stanley} for the case of matroids with cycle systems. 
In section~\ref{sec:bijections}, we use deletion/contraction trees to establish bijections between the bases of a matroid with cycle system and its set of coparking functions---each bijection is determined by a choice of ordering for the ground set of $M$. Finally, in section~\ref{sec:questions} we pose some open questions and suggest directions for further study.

\subsection{Acknowledgements.}
We thank the following people for their significant contributions to our understanding of cycle systems: Marlo Albers, Charlie Bruggemann, Lily Factora, Kolja Knauer, Yupeng Li, Sanay Sehgal, and Sherry Wang. We also thank Evan Huang and Kai Mawhinney for comments regarding Theorem \ref{thm:components}. We gratefully acknowledge our extensive use of the mathematical software SageMath \cite{sagemath} in this work. 

\section{Preliminaries}\label{sec: prelim}
We begin with some background material and relevant definitions. As a notational convenience, we usually omit brackets for singleton sets.

\subsection{Matroids}
We first review some basic notions of matroid theory, referring to  \cite{BrylawskiOxley1992} for more details. Recall that a \emph{matroid} $M = (E,{\mathcal I})$ on a finite \emph{ground set} $E=E(M)$ is a collection $\mathcal{I}=\mathcal{I}(M)$ of subsets of $E$ satisfying the following properties:
\begin{enumerate}
    \item $\emptyset\in\mathcal{I}$;
    \item If $X \in {\mathcal I}$ and $Y \subseteq X$ then $Y \in {\mathcal I}$;
    \item (Exchange property) If $X, Y \in {\mathcal I}$ and $|X| > |Y|$ then there exists $e \in X \setminus Y$ such that $Y \cup e \in {\mathcal I}$.
\end{enumerate}
 The elements of ${\mathcal I}$ are called the \emph{independent sets} of the matroid.  An independent set $B$ that is maximal (under inclusion) is called a \emph{basis}, and we let ${\mathcal B}(M)$ denote the set of bases of $M$. The number of elements in any (and hence every) basis of $M$ is called the \emph{rank} of the matroid, denoted $\rk(M)$. 
 More generally, the $\emph{rank}$ of a subset $S \subseteq E$, denoted $\rk(S)$, is the maximum cardinality of an independent subset of $S$.  The \emph{restriction} of $M$ to the subset $S$ is the matroid $M|S$ on the ground set $S$ whose independent sets consist of $X\in\mathcal{I}(M)$ such that $X\subseteq S$. Thus, $\rk(S)=\rk(M|S)$. The \emph{corank} of a matroid is $g(M)=|E|-\rk(M)$.

A subset of $E$ is \emph{dependent} if it is not independent, and a minimal dependent subset is called a \emph{circuit}. That is, $C\subseteq E$ is a circuit if and only if $C$ is dependent and all proper subsets of $C$ are independent. Note that a set is dependent if and only if it contains a circuit.  The collection of circuits of $M$ determines the independent sets of $M$, and hence a matroid can be defined by specifying the collection of circuits, denoted $\cir(M)$, rather than the independent sets. The circuits of a matroid satisfy the following characteristic properties:
\begin{enumerate}
    \item $\emptyset\notin\cir(M)$;
    \item If $C,C'\in\cir(M)$ and $C\subseteq C'$, then $C=C'$;
    \item\label{item:circuit3} (Circuit elimination property) If $C,C'\in\cir(M)$ with $C\neq C'$ and $e\in C\cap C'$, then there exists $D\in\cir(M)$ such that $D\subseteq (C\cup C')\setminus e$.
\end{enumerate}
In fact, we can replace property~(\ref*{item:circuit3}) with
\begin{enumerate}
    \item[(\ref*{item:circuit3})$^\prime$]  (Strong circuit elimination property) If $C,C'\in\cir(M)$ with $e\in C\cap C'$ and $f\in C\setminus C'$, then there exists $D\in\cir(M)$ such that $f\in D\subseteq (C\cup C')\setminus e $.
    
\end{enumerate}

If $B$ is a basis of a matroid $M$ and $e \in E\setminus B$, one can check that there exists a unique circuit $C$ with the property that $C \subseteq B \cup e$. The circuit $C = C(e,B)$ is called a \emph{fundamental circuit} with respect to $B$.

A \emph{loop} is an element $e\in E$ contained in no basis.  Equivalently, $\rk(e)=0$. A \emph{bridge} (or \emph{coloop}) is an element $e\in E$ contained in every basis.  Equivalently, $\rk(X\cup e) = \rk(X)+1$ for all subsets $X\subseteq E\setminus e$.  An element is a bridge precisely when it is not contained in any circuit.  A subset $K \subseteq E$ is a \emph{cycle} if it can be expressed as the union of circuits. Equivalently, $K$ is a cycle if and only if the restriction $M|K$ contains no bridges. 

Given a matroid $M = (E, {\mathcal I})$, the \emph{dual matroid} $M^*$ has ground set $E$ and bases given by the complements of the bases of $M$, so that ${\mathcal B}(M^*) = \{E \setminus B:B \in {\mathcal B}(M)\}$. 
Two matroids are \emph{isomorphic} if there is a bijection of their ground sets
inducing a bijection of independent sets.

We will also need the notions of deletion, contraction, and matroid minors. Suppose $S\subseteq E$. The \emph{deletion} $M\setminus S$ is the matroid on ground set $E \setminus S$, with independent sets $\{I \subseteq E \setminus S: I \in {\mathcal I}\}$.  If $S \subseteq E$ is an independent set, the \emph{contraction} $M/S$ is the matroid, again on ground set $E \setminus S$, with independent sets $\{I \subseteq E \setminus S: I \cup S \in {\mathcal I}\}$. A matroid $N$ is a \emph{minor} of the matroid $M$ if it can be constructed from $M$ by a sequence of deletion and contraction operations.

An important example of a matroid,
particularly relevant for us, comes from graph theory.  If $G$ is a finite graph
with vertex set $V(G)$ and edge set $E(G)$ (possibly with loops and multiple
edges), one defines the \emph{cycle matroid} $M(G)$ with ground set
$E = E(G)$ and independent sets given by forests (collections of edges that do
not contain a cycle). If $G$ is connected, the bases are spanning trees of $G$,
and hence the rank of $M(G)$ is given by $|V(G)| - 1$.
A \emph{graphic matroid} is a matroid that is isomorphic to the cycle matroid of a graph.  A matroid is \emph{cographic} if it is the dual of a graphic matroid.
We will often refer to the cycle matroid $M(G)$ simply as $G$.

If $M = M(G)$ is the cycle matroid of a graph $G$, then the deletion $M\setminus S$ is the cycle matroid on the graph obtained by removing from $G$ the elements (edges) in $S$. The contraction $M/S$ is the cycle matroid on the graph obtained by contracting each $e \in S$, i.e., removing $e = \{v,w\}$ and identifying $v$ and $w$ to be a single vertex. Note that any edge incident to either $v$ or $w$ in $G$ becomes incident to this new vertex.

Graphic and cographic matroids are special cases of linear matroids.  Let $E$ be a finite collection of vectors in some vector space $V$ over a field $F$.  To form the \emph{vector matroid} on $E$, declare subsets of $E$ to be independent sets of the matroid if they are linearly independent in $V$.  A \emph{linear matroid} is a matroid that is isomorphic to a vector matroid. A matroid is \emph{regular} if it can be realized as a vector matroid over any field $F$.
    
\subsection{Tutte polynomials and \texorpdfstring{$h$}{h}-vectors}
The collection of independent sets of a matroid forms a \emph{simplicial complex} $\Ind(M)$ called the \emph{independence complex} of $M$. Note that if $\rk(M) = d$, then $\Ind(M)$ has dimension $d-1$. Associated to any simplicial complex of dimension $d-1$ is its \emph{f-vector} $f = (f_{-1}, f_0, \dots, f_{d-1})$, where $f_{i-1}$ is the number of simplices of cardinality $i$.
 The $h$-vector of the independence complex of $M$ (which we will simply refer to as the \emph{$h$-vector of $M$}) encodes the same information as $f$ in a form that is more convenient, especially in algebraic contexts. 
 
 In particular, one can define the entries of $h = (h_0, \dots, h_d)$ according to the linear relation
\[\sum_{i=0}^d f_{i-1}(t-1)^{d-i} = \sum_{k=0}^d h_kt^{d-k}.\]
\noindent
The $h$-vector of a simplicial complex is related to a presentation of the Hilbert function of its Stanley-Reisner (face) ring, and in the case of matroids encodes combinatorial data regarding any shelling of its independence complex.  

One can also recover the $h$-vector of $M$ as an evaluation of its Tutte polynomial. For this, recall that if $M$ is a matroid on ground set $E$, the \emph{Tutte polynomial} of $M$ is the bivariate polynomial 
\[T_M(x,y) = \sum_{S \subseteq E} (x-1)^{\rk(M)-\rk(S)}(y-1)^{|S|-\rk(S)}.\]
Recall that $\rk(S)$ is the rank of the set $S$, as defined above. If one evaluates this expression at the value $y = 1$, we see that the term corresponding to $S$ is zero whenever $S$ is not independent. In addition, if $S$ is independent, we have $|S| = \rk(S)$. Since $\rk(M) = d$, we get 
\[T_M(x,1) = \sum_{I \in {\mathcal I}} (x-1)^{\rk(M)-|I|} = \sum_{i=0}^d f_{i-1}(x-1)^{d-i} = \sum_{k=0}^d h_kx^{d-k}.\]

We conclude that $T_M(x,1)$ is a polynomial in $x$ whose coefficients are the entries of the $h$-vector (in reverse order).
Also recall that if $M^*$ is the matroid dual to $M$, we have $T_{M^*}(x,y) = T_M(y,x)$. From this we see that $T_M(1,y)$ is a polynomial that recovers the $h$-vector of the dual $M^*$.

The following two properties, crucial for our results, determine the Tutte polynomial:
\begin{enumerate}
    \item If $M$ consists solely of bridges and loops, then $T_M(x,y)=x^by^{\ell}$ where $b$ is the number of bridges and $\ell$ is the number of loops.  In particular, if $E=\emptyset$, then $T_M(x,y)=1$.
    \item If $e\in E$ is neither a bridge nor loop, then $T_M(x,y)=T_{M\setminus e}(x,y)+T_{M/e}(x,y)$.
\end{enumerate}

\subsection{Multicomplexes and (pure) O-sequences}
We next review the notion of O-sequences and purity involved in the statement of Stanley's conjecture.
Recall that a \emph{multicomplex} $\Delta$ on a ground set $C = \{c_1, c_2, \dots, c_g\} $ is a finite collection of multisets of elements from $C$ that is closed under taking subsets: if $\sigma \in \Delta$ and $\tau \subseteq \sigma$, then $\tau \in \Delta$.  A multiset in $\Delta$ of cardinality $i$ (counting multiplicities) is called an $i$-face of $\Delta$.
The multisets in $\Delta$ which are maximal under inclusion are the \emph{facets} of $\Delta$.  A multicomplex is \emph{pure} if all its facets have the same cardinality (counting multiplicities).  The \emph{$f$-vector} of $\Delta$ is the sequence of positive integers $(f_{-1}, f_0, \dots, f_{d-1})$ where $f_i$ is the number of $i$-faces of $\Delta$ with cardinality $i+1$ and $d$ is the maximal cardinality of an element of $\Delta$.  A sequence of positive integers is said to be an \emph{O-sequence} if it is the $f$-vector of some multicomplex $\Delta$. The sequence is a \emph{pure O-sequence} if $\Delta$ can be chosen to be a \emph{pure} multicomplex.

We will always identify a multicomplex with a list of integer vectors as follows. Fix a ground set ordering $C=(c_1,\dots,c_g)$ for $\Delta$. Then the \emph{multiplicity vector} for a multiset $\alpha\in\Delta$ is the nonnegative integer vector $m(\alpha):=(a_1,...,a_g)\in\N^g$ where $a_i$ is the multiplicity of $c_i$ in $\alpha$,  yielding a bijection $m\colon\Delta\to\widetilde{\Delta}\subset\N^g$. Define a partial ordering on the set of nonnegative integer vectors $\N^g$ by comparing components: $a\leq b$ if $a_i\leq b_i$ for $i=1,\ldots,g$. The property of $\Delta$ of being closed under taking subsets translates to $\widetilde{\Delta}$ being an \emph{order ideal} in $\N^g$, that is, if $b\in\widetilde{\Delta}$ and $a\leq b$, then $a\in\widetilde{\Delta}$. Define the \emph{degree} of $a\in\N^g$ by $\deg(a)=\sum_{i=1}^ga_i$ and the \emph{degree vector} for $\widetilde{\Delta}$ by $d(\widetilde{\Delta}):=(d_0,d_1,\dots,d_s)$ where $d_i$ is the number of elements of $\widetilde{\Delta}$ of degree $i$ and $s$ is the maximal degree of an element of $\widetilde{\Delta}$.  Hence, the degree vector of $\widetilde{\Delta}$ is the $f$-vector of $\Delta$.  The multicomplex $\Delta$ is pure precisely when each maximal element of $\widetilde{\Delta}$ has the same degree.

The motivation for Conjecture \ref{conj:Stanley} comes from Stanley's work on Cohen-Macaulay complexes~\cite{StanleyCM}. Suppose $\Delta$ is a finite $(d-1)$-dimensional simplicial complex on vertex set $\{1,2, \dots,n\}$, and for a fixed field $\mathbbm{k}$, let $R = \mathbbm{k}[x_1, \dots, x_n]$ denote the polynomial ring over $\mathbbm{k}$ whose variables are the vertices of $\Delta$. One defines the \emph{Stanley-Reisner ideal} $I_\Delta \subset R$ to be the ideal generated by squarefree monomials $\{{\bf x}^a: a \notin \Delta\}$  corresponding to the nonfaces of $\Delta$. The quotient $R_\Delta=R/I_\Delta$ is called the \emph{Stanley-Reisner ring} of $\Delta$.

Stanley shows that the Krull dimension of $R_\Delta$ is $d$, and that the Hilbert function $H(R_\Delta; t)$ is given by 
\[H(R_\Delta; t) = \frac{\sum h_i t^i}{(1-t)^d},\]
where the $h_i$ are the entries of the $h$-vector of $\Delta$. It follows from this (and a result of Stanley that he traces back to Macaulay) that if $\Delta$ is a \emph{Cohen-Macaulay} complex, then its $h$-vector is an O-sequence. One can characterize the Cohen-Macaulay property algebraically (namely that $R_\Delta$ has depth equal to $d$), and also topologically: For all faces $F \in \Delta$ (including $F = \emptyset$) we have $\tilde H_i(\lk F; \mathbbm{k}) = 0$ whenever $i \neq \dim \lk F$. Here $\lk F := \{G \in \Delta: F \cap G = \emptyset, F \cup G \in \Delta\}$ is the \emph{link} of $F$ in $\Delta$. 
 
 It is known that the independence complex of a matroid is Cohen-Macaulay, and hence its $h$-vector is an O-sequence. As matroids constitute a special class of Cohen-Macaulay complexes (for example, they are vertex decomposable), one expects that their $h$-vectors should enjoy further properties.
 
\section{Cycle systems}
\label{sec: cycle systems}

In this section we provide the definition of a cycle system on a matroid, and discuss several examples.  The construction depends on the notion of the \emph{unique union operator}, denoted $\ast$, which takes a collection of
subsets of some fixed set and returns the elements that appear in exactly one of the subsets.  More precisely we have the following.

\begin{defn}\label{defn:uniqueunion}
Let $S$ be a set, and suppose~$\mathcal{A}=\{A_{1},\ldots,A_{g}\}$, where each~$A_{i} \subseteq S$.  For each~$\sigma\subseteq[g]$, the \emph{unique union} of the subset $\textstyle\{A_{i}:i\in
\sigma\}$ of~$\mathcal{A}$, denoted~$\mathcal{A}_{\sigma}$, is the set of elements of~$S$ that appear in exactly one
of the~$A_{i}$:
\[
  \mathcal{A}_{\sigma}:=\ast\{A_{i}:i\in
    \sigma\}:=\{e\in\textstyle\bigcup_{i\in\sigma}A_{i}:
  \text{there exists a unique~$i\in \sigma$ such that~$e\in A_{i}$}\}.
\]
\end{defn}

One can see that the unique union has the following properties:
\begin{itemize}
  \item If $i\in\tau \subseteq \sigma\subseteq[g]$, then $\mathcal{A}_{\sigma}\cap
    A_{i}\subseteq\mathcal{A}_{\tau}\cap A_{i}$.
  \item If~$\sigma,\tau \subseteq [g]$, then
    $\mathcal{A}_{\sigma\cup\tau}\subseteq \mathcal{A}_{\sigma}\cup \mathcal{A}_{\tau}$. 
    \item If $\mathcal{S},\mathcal{T}\subseteq\mathcal{A}$, then $\ast(\mathcal{S}\cup\mathcal{T})\subseteq \ast(\mathcal{S}\cup\{\ast \mathcal{T}\})$.
\end{itemize}
However, note that if~$A,B,C$ are sets with $C\subseteq B$, it is not necessary
that one of $\ast\{A, B\}$ or $\ast\{A, C\}$ is contained in the other.  For
example, let~$A = \{1,2\}$, $B=\{2,3,4\}$, and~$C=\{3\}$.  Then~$\ast\{A, B\} =
\{1,3,4\}$ and $\ast\{A,C\}=\{1,2,3\}$. 

\begin{remark}
The unique union of two sets is their symmetric difference: $\ast\{A,B\}=A\Delta
B$.  When applied to more than two sets, the unique union is still a subset of the symmetric difference, but not generally equal. For example, if $A=\{1,2\}$, $B=\{1,3\}$,
and~$C=\{1,4\}$, then
\[
 \ast\{A,B,C\}=\{2,3,4\}\neq A\Delta B\Delta C :=(A\Delta B)\Delta
 C=\{1,2,3,4\}.
\]
Put another way, $\{2,3,4\}=\ast\{A, B, C\}\subsetneq\ast\{\ast\{A, B\},C\}=\{1,2,3,4\}$. 
\end{remark}

\begin{defn}\label{defn:cyclesystem}
 A \emph{cycle system} for a matroid~$M$ is a collection of cycles $\cs = \{C_{1},\ldots,C_g\}$ such that
 \begin{enumerate}
   \item $g = g(M)$, and
   \item\label{item:unique union property} for each $\emptyset\neq\sigma\subseteq[g]:=\{1,\ldots,g\}$, the unique union
     $\cs_{\sigma}$ is dependent.
 \end{enumerate}
   A \emph{circuit system} is a cycle system consisting of circuits, and a
   \emph{fundamental circuit system} is a cycle system whose elements are the
   fundamental circuits with respect to a base for~$M$.
\end{defn}

\begin{remark}
     By Corollary~\ref{cor:circuit systems}, below, a matroid has a cycle system if and only if it has a circuit system. Nevertheless, we allow non-circuit cycle systems since they may arise naturally through the operations of contraction essential for our results (see~Proposition~\ref{prop:del-con cs}).
\end{remark}

In general,  say that a collection $\mathcal{A}=\{A_1,...,A_k\}$ of subsets of $M$ has the \emph{unique union property} if the unique union of each nonempty subset of $\mathcal{A}$ is a dependent set.
The next result implies that a cycle system is a maximal set of cycles satisfying the unique union property.

\begin{prop}\label{prop:max size for cs}
    Let $\mathcal{A}=\{A_1,...,A_k\}$ be a collection of subsets of $M$ with the unique union property. Then $k\leq g$.
\end{prop}

\begin{proof}
If $g=0$, then $M$ has no dependent sets, and thus $k=0\leq g$.  If $k,g\geq1$, then $\mathcal{A}_{[k]}$ is dependent and thus contains a non-bridge element $e$. Without loss of generality, suppose that $e\in A_k$.  Then $\{A_1,...,A_{k-1}\}$ is a set of subsets of $M\setminus e$ with the unique union property, and the corank of $M\setminus e$ is $g-1$.  The result follows by induction.
\end{proof}

\subsection{Examples and remarks} \label{sec:examples}
\begin{enumerate}[leftmargin=*]
    \item If $G$ is a graph with vertices $\{v_0,v_1,\dots,v_n\}$, then the cographic matroid $M(G)^*$ has ground set $E(G)$ and circuits given by the minimal edge-cuts of $G$ (these are called \emph{cocircuits} of $G$). If $C_i$ is the set of edges of $G$ incident to vertex $v_i$, then $\{C_1,\dots,C_n\}$ forms a circuit system for $M(G)^*$. 

    \item\label{item:cone system} If $G$ is a graph embedded in the plane, then the collection of circuits defined by its bounded faces forms a circuit system for the graphic matroid $M(G)$. Indeed, these sets are the same as the cocircuits of the dual graph $G^*$ defined by the cuts determined by the corresponding vertices. We have seen above that this collection forms a cycle (in fact, circuit) system. This was the motivating example discussed in the introduction (see Figure~\ref{fig:intro}).
     \item\label{example:cone} If $G$ is a graph on $n$ vertices $\{v_1,\dots,v_n\}$, then the \emph{cone} over $G$ is the graph $\textrm{Cone}(G)$ obtained by adding a vertex $v$ together with an edge $f_i$ from $v$ to each $v_i$.  We call $\textrm{Cone}(G)$ a \emph{coned} graph with \emph{cone vertex}~$v$. For each edge $e=v_iv_j$ in the original graph, consider the circuit $C_e=\{f_i,e,f_j\}$ passing through the new vertex $v$. Then the collection $\cs=\{C_e \, | \, e\in E(G)\}$ provides a circuit system for $M(\textrm{Cone}(G))$. Indeed, the corank of $\textrm{Cone}(G)$ is $|E(G)|$, so $\cs$ has the correct size. Now suppose that $\sigma\subseteq E(G)$ is a nonempty subset of edges, and view $\sigma$ as a subgraph of $G$. If $\sigma$ contains a circuit $C$, then $C\subseteq \cs_\sigma$, which is therefore dependent. On the other hand, if $\sigma$ is acyclic, then choose a path $P$ in $\sigma$ between two leaf vertices $v_i$ and $v_j$ of $\sigma$. Then the unique union $\cs_\sigma$ contains the circuit $\{f_i,P,f_j\}$, and is therefore dependent.
     \item A computer search shows the bipartite graph $K_{3,3}$ has no circuit system. By Corollary~\ref{cor:K33-free}, every graph that is $K_{3,3}$-free has a cycle system. However, the converse does not hold: $K_6$ is the cone over $K_5$ and hence has a cycle system, but $K_{3,3}$ appears as a minor of $K_6$, obtained by deleting edges. 
     \item The uniform matroid $U(2,4)$ has no cycle system. Recall that this is the rank-2 connected matroid with ground set $[4]$ and circuits consisting of the four subsets of size $3$. Any circuit system would consist of two distinct circuits, and by symmetry we can assume that these are $C_1=\{2,3,4\}$ and  $C_2=\{1,3,4\}$. Then the unique union $\ast\{C_1, C_2\}=\{1,2\}$ is independent, showing that these do not form a circuit system. By connectivity (Theorem \ref{thm:circuits}), $U(2,4)$ has no cycle system. In fact, most uniform matroids $U(m,n)$ have no cycle systems. Indeed, cycle systems exist only when $m\in\{0,1,n-1,n\}$, as can be established via induction using Proposition~\ref{prop:del-con cs}. Note that a matroid $M$ is binary if and only if it is $U(2,4)$-minor-free.
     \item Every example of a matroid with a cycle system that we know of is regular. In particular, we have no example of a non-binary matroid with a cycle system (see Conjecture~\ref{conj: nonbinary}).
     \item A connected graph is \emph{biconnected} if it has no \emph{cut-vertex}, i.e., a vertex whose removal, along with incident edges, disconnects the graph. A \emph{biconnected component} of a graph is a maximal biconnected subgraph.  Every graph can be uniquely decomposed into biconnected components, with each connected component of the graph consisting of biconnected components attached via cut-vertices.  A graph has a circuit system if and only if  
      each of its biconnected components has a circuit system (see~Corollary~\ref{cor:circuit systems}). 
      \item\label{item:example-subdivision} 
      Suppose a graph $G'$ is obtained from a graph $G$ by \emph{subdividing an edge}, i.e., by replacing an edge of $G$ with a path of three or more vertices. The cycles of $G'$ and $G$ are in bijection, and $G'$ has a circuit system if and only if $G$ has a circuit system. (For instance, a cycle graph can be obtained by subdividing the single edge of a loop.) Similarly, removing vertices of degree one preserves both the cycle structure and the existence of circuit systems. Hence, for the purpose of finding circuit systems, one may assume without loss of generality that all vertices have degree at least three (if we allow for graphs with multiple edges and loops).
     \item See the \hyperref[appendix]{Appendix} for further discussion of circuit systems on graphs with at most eight vertices and for examples of matroids with circuit systems but no fundamental circuit system.
     \item\label{item:cs-finding algorithm} (\textsc{circuit system search algorithm}.) Let $M$ be a matroid of corank~$g$ with circuit-set $C(M)$.  For each $k\geq 1$, let $L_k$ be the collection of all $k$-element subsets of $C(M)$ with the unique union property. The elements of $L_g$ are exactly the circuit systems for $M$.  The algorithm we use to find circuit systems successively builds the $L_k$ starting with $L_1=\{\{C\}:C\in C(M)\}$.  Given $L_k$ one may build $L_{k+1}$ as follows. For each $A=\{C_1,...,C_k\}\in L_k$, take $C\in C(M)\setminus A$.  We must decide whether $A':=A\cup\{C\}$ has the unique union property.  To do that, first check whether $\ast (A')$ is dependent.  Then, for each $i$, check whether $\ast(A'\setminus \{C_i\})$ is in $L_k$.  These latter steps do not require re-checking the dependence condition, and we only need store $L_k$ to find $L_{k+1}$. Also, once $k$ gets close to $g$, the set $L_k$ is typically much smaller than the collection of all $k$-element subsets of $C(M)$.  One side effect of this algorithm is that one finds \emph{all} circuit systems almost as quickly as finding a single one (which one would do by stopping after finding the first element of $L_g$). 
\end{enumerate}

\subsection{Properties}

In this section we establish various properties of cycle systems. In what follows, we let~$M$ be a matroid with cycle system $\cs = \{C_{1},\ldots,C_g\}$. Our first observation is that such structures behave well with respect to certain deletions and all contractions.

\begin{prop}\label{prop:del-con cs}\ 
\leavevmode  
\begin{enumerate}
  \item\label{item:del-con cs 1} Let $e\in\cs_{[g]}\cap C_i$.  Then 
    \[
    \cs'=\cs\setminus\{C_{i}\}= \{C_1,\dots,C_{i-1},\widehat{C_i}, C_{i+1}, \dots, C_g\}    \]
    is a cycle system for $M\setminus e$.
  \item\label{item:del-con cs 2} Let $e$ be any non-loop of $M$.  Then 
    \[
      \cs'':=\{C_1\setminus e,...,C_g\setminus e\}
    \]
    is a cycle system for $M/e$.
\end{enumerate}
\end{prop}

\begin{proof} (\ref{item:del-con cs 1}) First note that the elements
  of~$\cs\setminus\{C_{i}\}$ are cycles of $M\setminus e$
  since~$e\not\in C_{j}$ for~$j\neq i$.  Next, for each
  nonempty~$\sigma\subset[g]\setminus{i}$, we
  have~$\cs'_{\sigma}=\cs_{\sigma}$, which is dependent
  in~$M\setminus e$.  The cardinality of~$\cs'$ is correct
  since~$g(M\setminus e)=g(M)-1$.  
  \smallskip

  \noindent(\ref{item:del-con cs 2}) We first verify that $\cs''$ consists of cycles.  Let $f\in
  E(M)\setminus\{e\}$ be a non-bridge element of~$M$.  It suffices to show
  that~$f$ is a non-bridge of~$M/e$. Since~$f$ is a non-bridge of~$M$,
  there exists a basis~$B$ for~$M$ not containing~$f$. If~$e\in B$,
  then~$B\setminus\{e\}$ is a basis for $M/e$ not containing~$f$,
  and hence $f$ is a non-bridge of~$M\setminus e$.  Otherwise, since~$e$ is not
  a loop, there exists a basis~$B'$ for~$M$ containing~$e$.  By the basis
  exchange property of matroids, there exists~$x\in B$ such
  that~$B'':=(B\setminus\{x\})\cup\{e\}$ is a basis for~$M$.
  Then~$B''\setminus\{e\}$ is a basis for~$M/e$ not containing~$f$.
  Thus,~$f$ is not a bridge of~$M/e$.

  Now let $\emptyset\neq\sigma\subseteq[g]$. Since $\cs$ is a cycle
  system on $M$, we know $\cs_\sigma$ contains a cycle $C$. It follows
  $\cs_\sigma''\supseteq C\setminus e$, and $C\setminus e$ is dependent
  in $M/e$ since $(C\setminus e)\cup \{e\}$ is dependent in $M$. The cardinality
  of $\cs''$ is correct since the fact that~$e$ is a non-loop implies
  $g(M/e)=g(M)$.
\end{proof}

\begin{prop}\label{prop:every circuit}
  Let~$D$ be a circuit of~$M$.  Then there exists~$\sigma\subseteq[g]$ such
  that~$\cs_{\sigma}=D$.
\end{prop}
\begin{proof}  The proof is by induction on~$g$.  If $g=1$, then~$M$ has a
  unique circuit.  (To see this, choose a basis~$B$ for~$M$.  Since~$g=1$, we
  have~$B = E(M)\setminus\{e\}$ for some element~$e$.  Then~$B\cup\{e\}$
  contains a unique circuit---the fundamental circuit for the
  pair~$(B,e)$---and~$B\cup\{e\}=E(M)$.) So, in this case, $\cs=\{D\}$
  and~$\sigma=\{1\}$.  

  Now suppose~$g>1$. If~$\cs_{[g]}$ contains an element~$e\not\in D$,
  without loss of generality, let~$e\in C_{g}$.  Consider the cycle system
  $\cs'=\{C_{1},\ldots,C_{g-1}\}$ on the matroid~$M'=M\setminus e$
  (see~Proposition~\ref{prop:del-con cs}).  Since~$D$ is a circuit in~$M'$,
  and~$g(M')=g-1$, by induction there exists~$\sigma\subseteq[g-1]$ such
  that~$\cs'_{\sigma}=D$. The result follows in this case
  since~$\cs'_{\sigma}=\cs_{\sigma}$.  The alternative is
  that~$\cs_{[g]}\subseteq D$.  Since~$\cs$ is a cycle system,
  there exists a circuit~$C\subseteq\cs_{[g]}$. We
  have~$C\subseteq\cs_{[g]}\subseteq D$. Since~$C$ and~$D$ are
  circuits,~$C=D$, which implies~$\cs_{[g]}=D$.
\end{proof}

\begin{cor}\label{cor:non-bridges}
  An element~$e$ of $M$ is a non-bridge if and only if there exists~$i$ such
  that~$e\in C_{i}$.
\end{cor}
\begin{proof}
  The result follows immediately from Proposition~\ref{prop:every circuit}
  since an element is a non-bridge if and only if it is contained in
  a circuit.  
\end{proof}

\begin{cor}\label{cor:loops} Let~$S = \cs_{[g]}\cap C_{i}$.   Then there
  exists a circuit in $C_{i}$ containing $S$.  If~$S$ contains a loop~$\ell$,
  then~$S=\{\ell\}$.
\end{cor}
\begin{proof}
  Suppose there is no circuit in~$C_{i}$ containing~$S$.  Since~$C_{i}$ is a
  union of circuits, there exist $e,e'\in S$ and circuits
  $C,C'$ contained in~$C_{i}$ such that~$e\in C \setminus C'$ and $e'\in
  C'\setminus C$.  Consider the matroid~$\tilde{M}:=M\setminus e$ with cycle
  system $\tilde{\cs}:=\cs\setminus\{C_{i}\}$
  (see~Proposition~\ref{prop:del-con cs}). Then~$e'$ is a non-bridge
  in~$\tilde{M}$ since it is contained in $C'$, which is a circuit
  of~$\tilde{M}$.  Hence, by Corollary~\ref{cor:non-bridges}, we know $e'$ is in
  some cycle of $\tilde{\cs}$, i.e., $e'\in
  \cs_{[g]\setminus\{i\}}$, which is not true.  

  Hence, there must be a circuit~$C\subseteq C_{i}$ containing~$S$.  If~$S$
  contains a loop~$\ell$, then since $\ell$ is a circuit, we must
  have~$\{\ell\}=C$, and thus,~$S=\{\ell\}$.
\end{proof}
  
\subsection{Circuit systems}\label{subsection:circuit systems}
By definition a cycle system $\cs$ on a matroid $M$ consists of a collection of cycles. We will see that under certain circumstances these cycles must be circuits, i.e., $\cs$ is a circuit system. Furthermore, we see that if $M$ is any matroid that has a cycle system, then it in fact admits a circuit system.  This observation allows one to narrow a computer search for such structures. Recall that a matroid is \emph{connected} if for each pair of groundset elements, there is a circuit containing those elements.\footnote{Let $G$ be a loopless graph with at least three vertices and no isolated vertex.  Then $G$ is biconnected if and only if its cycle matroid $M(G)$ is connected \cite[Proposition 4.1.1]{Oxley}.}  

\begin{thm}\label{thm:circuits}
Suppose $M$ is a connected matroid with cycle system $\cs$. Then each element of $\cs$ is a circuit.
\end{thm}
\begin{proof} We prove this by induction on~$|E(M)|$.  If~$|E(M)|=1$, then
  $E(M)=\{e\}$ where~$e$ is either independent or a loop. In the former
  case,~$\cs=\emptyset$, and in the latter, $\cs=\{e\}$.  The result holds in
  either case. So suppose~$|E(M)|>1$.  Since~$M$ is connected, it has at least one
  circuit, and hence~$g\geq 1$.  Without loss of generality, suppose there is an element~$e_{g}\in \cs_{[g]}\cap C_{g}$.  
  The element~$e_{g}$ is not a bridge since it is contained in~$C_{g}$, and it
  is not a loop since~$M$ is connected with more than one element.

  Since~$M$ is connected, either~$M\setminus e_{g}$ or $M/e_{g}$ is connected
  (\cite[Proposition 4.3.1]{Oxley}).  By Proposition~\ref{prop:del-con cs}, we
  know $\cs = \{C_{1},\ldots,C_{g-1}\}$ is a cycle system
  for~$M\setminus e_{g}$.  If $M\setminus e_{g}$ is connected, then by
  induction, $C_{1},\ldots,C_{g-1}$ are circuits in $M\setminus e_{g}$, and
  hence circuits in $M$.  On the other hand, suppose~$M/e_{g}$ is connected.
  Proposition~\ref{prop:del-con cs} implies~$\{C_{1},\ldots, C_{g-1},
  C_{g}\setminus e_g\}$ is a cycle system, and by induction, it consists of
  circuits of~$M/e_{g}$.  It again follows that~$C_{1},\ldots,C_{g-1}$ are
  circuits in~$M$.

  At this point, we may assume $C_{1},\ldots,C_{g-1}$ are circuits, and it
  remains to be shown that~$C_{g}$ is also a circuit.  For the sake of
  contradiction, suppose it is not.  Let $D$ be a circuit of~$C_{g}$
  containing~$e_{g}$, and let~$e\in C_{g}\setminus D$.  Let~$C\subset C_{g}$
  be a circuit containing~$e$.  By the generalized circuit exchange property,
  if~$e_{g}\in C\cap D$, we may replace~$C$ by a circuit~$C'$ such that~$e\in
  C'\subseteq (C\cup D)\setminus\{e_{g}\}$.  So we may assume $e_{g}\not\in C$.  Since~$M$ is
  connected, there is a circuit~$K$ with~$e,e_{g}\in K$.  Pick subsets
  $\sigma,\tau\subseteq[g]$ such that~$\cs_{\sigma}=K$ and~$\cs_{\tau}=C$.
  Since~$e_{g}\in\cs_{[g]}\cap C_{g}$, it follows that~$g\in \sigma$
  and~$g\not\in\tau$.  But~$\cs_{\tau}=C\subset C_{g}$,
  which implies~$\cs_{\tau\cup\sigma}\cap C=\emptyset$.  That is because every element
  of~$C$ appears in at least two distinct cycles indexed by elements
  of~$\tau\cup\sigma$: namely,
  in $C_{i}$ for some $i\in \tau$ and in~$C_{g}$. We have 
  $\cs_{\tau\cup\sigma}\subseteq
  \cs_{\tau}\cup\cs_{\sigma}=C\cup K$,
  and~$\cs_{\tau\cup\sigma}$ contains no elements of~$C$.  Therefore,  
  \[
    \cs_{\sigma\cup\tau}\subseteq K\setminus\{e\},
  \]
  implying~$\cs_{\sigma\cup\tau}$ is independent, contradicting the fact
  that $\cs$ is a cycle system.
\end{proof}

We recall the notion of the decomposition of a matroid into connected components \cite[Chapter 4]{Oxley}. 
A \emph{connected component} of a matroid $M$ is a subset $T$ of the groundset that is maximal (under inclusion) with respect to the property that $M|T$ is connected, i.e., given any two distinct elements $x,y\in T$, there is a circuit contained in $T$ containing both $x$ and $y$. 
Let $M_1,...,M_n$ be matroids with disjoint groundsets. The \emph{direct sum} $M=M_1\oplus\cdots\oplus M_n$ is the matroid with groundset $E(M)=\bigcup_{i=1}^nE(M_i)$ and independent sets $I(M)=\{I_1\cup\cdots\cup I_n: I_i\in I(M_i)\text{ for all $i$}\}$. If $T_1,...,T_k$ are the connected components of a matroid $N$, let $N_i=N|T_i$ for $i=1,...,k$.  Then $N=N_1\oplus\cdots\oplus N_k$ is the \emph{decomposition of $N$} into connected components.  It is unique up to a reordering of the $T_i$.

\begin{lemma}\label{lemma:decomp}
  Let $M$ be a matroid with cycle system $\cs = \{C_1,...,C_g\}$.
  Suppose there exists~$i\in [g]$ such that~$\cs_{[g]}\cap C_{i}$ contains a circuit~$C$.  Then $M=C\oplus
  (M\setminus C)$.  Further, $\{C_{j}:j\in[g]\setminus\{i\}\}$ is a cycle system
  for~$M\setminus C$.
\end{lemma}
\begin{proof}
  Without loss of generality, $i=g$. Consider an arbitrary circuit $D$ that has nonempty intersection with $C$; we wish to show that $D=C$. Choose $x\in C\cap D$. By Proposition~\ref{prop:every circuit}, there
  exists~$\sigma\subseteq[g]$ such that~$\cs_{\sigma}=D$. Since~$x\in
  \cs_{[g]}\cap C_{g}$, it must be that~$g\in \sigma$.  However, if that is the
  case,
 \[
   C\subseteq \cs_{[g]}\cap C_{g}\subseteq\cs_{\sigma}\cap
   C_{g}\subseteq \cs_{\sigma}=D.
 \]
Since~$C$ and~$D$ are circuits, it follows that $C=D$. Thus, it must be that~$M=C\oplus (M\setminus C)$.
 Since~$C\subseteq \cs_{[g]}\cap C_{g}$, we have~$C\cap
 C_{j}=\emptyset$ for all~$j\neq g$.  Since $M=C\oplus (M\setminus C)$, we have
 $g(M\setminus C)=g-1$. It then easily follows that~$\{C_{j}:j\in[g-1]\}$ is a cycle system
  for~$M\setminus C$.
\end{proof}
  
\begin{lemma}\label{lemma:straightening}
  Let $M$ and $N$ be matroids, and suppose that $M\oplus N$ has a cycle system
  of the form 
  \[
    \mathcal{E}=\{C_{1},\ldots,C_{g_M}, K_{1}\cup D_{1},\ldots,K_{g_N}\cup D_{g_N}\}
  \]
  where each $C_{i}$ is a cycle in $M$, each~$K_{i}$ is either a cycle in~$M$ or
  empty, and each~$D_{i}$ is a cycle in~$N$. Then
  $\cs:=\{C_{1},\ldots,C_{g_{M}}\}$ is a cycle system for $M$
  and~$\mathcal{D}:=\{D_{1},\ldots,D_{g_{N}}\}$ is a cycle system for $N$.
\end{lemma}
\begin{proof} The set~$\cs$ has the unique union property since it is a
  subset of $\mathcal{E}$. Hence, it is a cycle system for $M$.  Now
  consider~$\mathcal{D}$. Given $\emptyset\neq\sigma\subseteq[g_N]$, we must
  show $B:=\mathcal{D}_{\sigma}$ is dependent. Let
  $\mathcal{T}=\{K_{i}:i\in\sigma\}$ and~$A=\ast \mathcal{T}\subseteq E(M)$
  so that 
  \[
    *\{K_i\cup D_{i}:i\in\sigma\}=A\cup B.
  \]
  If~$A$ is independent, then since
  $\mathcal{E}$ is a cycle system, $B$ must be dependent, as required.  On the
  other hand, if~$A$ is dependent, then by Proposition~\ref{prop:max size for cs}, the
  set~$\{C_{1},\ldots,C_{g_{M}},A\}$ cannot have the unique union
  property as its cardinality exceeds~$g_M$. So there exists a
  subset~$\mathcal{S}\subseteq\cs$ such that~$Y:=\ast(\mathcal{S}\cup
  \{A\})$ is independent. Since~$\mathcal{E}$ is a cycle system, the set
  \[
    *(\{K_{i}\cup D_{i}:i\in \sigma\}\cup
    \mathcal{S})=\ast(\mathcal{S}\cup\mathcal{T})\cup B
  \]
  is dependent.  However,
  $\ast(\mathcal{S}\cup\mathcal{T})$ is independent since $\ast(\mathcal{S}\cup
  \mathcal{T})\subseteq\ast(\mathcal{S}\cup\{\ast \mathcal{T}\})=Y$.
  Therefore,~$B$ must be dependent in this case, too. Thus,~$\mathcal{D}$ is a cycle system
  for~$N$.
\end{proof}

\begin{thm}\label{thm:components} Let $W=\oplus_{i=1}^{k}M_{i}$ be the direct
  sum decomposition of the matroid~$W$ into connected components, and
  let~$\mathcal{E}$ be a cycle system for~$W$.  Up to a reordering of the
  components of~$W$, there are cycle systems~$\cs^{(i)}$ for $M_{i}$ for
  each~$i$ such that~$\mathcal{E}$ is the disjoint union of sets
  $\widetilde{\cs}^{i}$ of the form
  \[
    \widetilde{\cs}^{i} =\{K_C\cup C:C\in \cs^{(i)}\}
  \]
  where each~$K_{C}$ is either empty or a cycle contained in $\oplus_{j=1}^{i-1}M_{j}$.  In
  particular, $\widetilde{\cs}^{(1)}=\cs^{(1)}$.
\end{thm}

\begin{proof} The theorem is true if~$W$ has only one component.  So, by
  induction using Lemma~\ref{lemma:straightening}, it suffices to show that~$\mathcal{E}$
  contains a cycle system for some component of~$W$.  We may suppose that~$W$
  has at least two components and no components of corank~$0$.  We use induction
  on the cardinality of~$E(W)$. Let $A$ be a circuit contained in
  $*\mathcal{E}$, and let~$e\in A$. Without loss of generality, $A\subseteq
  E(M_{1})$.  Say $e\in C\cup K\in\mathcal{E}$ where $C$ is a cycle of $M_{1}$ and~$K$ is
  either empty or a cycle of $ W\setminus M_{1}$. 

  If~$A\subseteq C\cup K$, then by Lemma~\ref{lemma:decomp}, we have~$W=A\oplus(W\setminus A)$.
  By connectedness,~$M_{1}=A=C$, and~$\mathcal{E}':=\mathcal{E}\setminus(C\cup K)$
  is a cycle system for~$W\setminus A=\bigoplus_{i=2}^{k}M_{i}$.  The result
  then follows by induction:~$\mathcal{E}'$ contains a cycle system for some
  component of~$W\setminus A$, and hence, $\mathcal{E}$ contains a cycle system
  for some component of~$W$.  So assume~$A\not\subseteq C\cup K$.  In
  particular,~$e$ is not a loop.

  Since~$M_{1}$ is connected, at least one of $M_{1}\setminus e$ and $M_{1}/e$
  is connected.  Let~$g$ denote the corank of~$M_{1}$. First, suppose~$M_{1}/e$
  is connected, and let~$\mathcal{E}''$ be the usual restriction
  of~$\mathcal{E}$ to a cycle system on~$W/e
  =(M_{1}/e)\oplus(\bigoplus_{i=2}^{k}M_{i})$, obtained from~$\mathcal{E}$ by
  replacing~$C\cup K$ by $(C\setminus e)\cup K$.  By induction, $\mathcal{E}''$ contains
  a cycle system for some component of~$W/e$.  If that component is
  not~$M_{1}/e$, we are done.  Otherwise, there exist
  cycles~$C_{1},\ldots,C_{g}\in\mathcal{E}''$ forming a cycle system
  for~$M_{1}/e$.  For each~$i$, either $C_{i}$ or $C_{i}\cup e$ is an element
  of~$\mathcal{E}$.  Therefore, $\mathcal{E}$ has~$g$ elements contained in
  $E(M_{1})$.   Since~$\mathcal{E}$ is a cycle system, these~$g$ elements
  satisfy the unique union condition and hence form a cycle system for~$M_{1}$,
  as required.
  (In fact, since~$e$ is a non-bridge, it must be contained in one of these~$g$
  elements.  Since $C\cup K$ is the only element of~$\mathcal{E}$ containing~$e$,
  it follows that~$K=\emptyset$.)

  Finally, suppose that~$M_{1}\setminus e$ is connected, and
  let~$\mathcal{E}'=\mathcal{E}\setminus(C\cup K)$ be the usual restriction to a
  cycle system on~$W\setminus e=(M_{1}\setminus e)\oplus
  (\bigoplus_{i=2}^{k}M_{i})$.  By induction, $\mathcal{E}'$ contains a cycle
  system for one of the components of~$W\setminus e$.  If it is one of the
  $M_{i}$ with $i\geq 2$, then we are done.  Otherwise, there exist
  cycles~$C_{1},\ldots,C_{g-1}\in\mathcal{E}'\subset\mathcal{E}$, contained
  in~$M_{1}$ and forming a cycle system for~$M_{1}\setminus e$. Let~$f\in
  A\setminus C$. Then, since~$M_{1}$ is connected,~$f$ is not a loop.  If
  $M_{1}/f$ is connected, we are done, as argued earlier.
  Otherwise,~$M_{1}\setminus f$ is connected, and arguing as at the beginning of
  this paragraph, we may assume that there exist
  circuits~$C_{1}',\ldots,C_{g-1}'\in\mathcal{E}$ forming a
  cycle system for~$M_{1}\setminus f$.  Since~$e$ is contained in the cycle $C$ of $M_1\setminus f$, it is a non-bridge element
  of~$E(M_{1}\setminus f)$, so~$e\in C_{i}'$ for some~$i$ (see Cor.~\ref{cor:non-bridges}). However,
  the only element of~$\mathcal{E}$ containing~$e$ is $C\cup K$.  Thus, $C_{i}'=C\cup K$
  for some~$i$, which forces~$K=\emptyset$.  It follows
  that~$\cs=\{C_{1},\ldots,C_{g-1},C\}$ is a subset of~$\mathcal{E}$
  with cardinality~$g$ and consisting of cycles in $M_{1}$.  Since~$\mathcal{E}$
  is a cycle system, $\cs$ satisfies the unique union condition and
  hence forms a cycle system for~$M_{1}$, which completes the proof. 
\end{proof}

\begin{cor}\label{cor:circuit systems}
Let $M$ be a matroid.
\begin{enumerate}
    \item\label{cor:cir1} $M$ has a cycle system if and only if each of its connected components has a cycle system.
    \item\label{cor:cir2} If $M$ has a cycle system, then it has a circuit system.
\end{enumerate}
\end{cor}
\begin{proof}
    For part~(\ref{cor:cir1}), note that if each of the connected components of $M$ has a cycle system, then the union of these cycle systems forms a cycle system for $M$.  
    Conversely, if we suppose $M$ has a cycle system, then each of its components has a cycle system by Theorem~\ref{thm:components}. Part~(\ref{cor:cir2}) then follows from Theorem~\ref{thm:circuits}.
\end{proof}

Circuit systems also have connections to an underlying vector space associated to a matroid. To recall the relevant definitions, suppose $M$ is a matroid on ground set $E$, and let $W$ be the $\mathbb{F}_2$-vector space with basis $E$. Then the \emph{circuit space} of $M$ is the subspace of $W$ spanned by the indicator vectors for circuits $C$ in $M$. We then have the following observation.

\begin{prop}\label{prop:circuitspace}
    Suppose $\cs=\{C_i\}$ is a circuit system for a matroid $M$. Then the corresponding indicator vectors $w_i$ are linearly independent in the circuit space $W$. In particular, if $M$ is binary, then the collection of indicator vectors $\{w_i\}$ for $\cs=\{C_i\}$ forms a basis for $W$.
    \end{prop}

    \begin{proof}
    First note that $w_i$ is indeed an element of $W$, since we are assuming $C_i$ is a circuit. The claim then follows immediately from the unique union property: if $0=\sum_{i\in\sigma}w_i$ is a nontrivial linear dependence, then the set $\cs_\sigma$ is dependent, and in particular nonempty. This means that there exists an index $i\in\sigma$ such that the corresponding indicator vector $w_i$ has a $1$ in a component where all other $w_j$ have zeros. But this contradicts the supposed linear dependence. 
    
    For the last statement, a matroid $M$ is binary if and only if the dimension of the circuit space is equal to the corank $g$.
    \end{proof}

    We note that if $\cs$ is only assumed to be a \emph{cycle} system, then the vectors $w_i$ corresponding to~$C_i$ may not lie in $W$. Also, for non-binary matroids the dimension of the circuit space is strictly larger than the corank.  While this result is known (see, e.g., \cite{Garamvolgyi} for a proof), we have not found it stated explicitly in the literature. In particular if $M$ is non-binary, a circuit system is never a basis for the circuit space.

\subsection{Parallel connections and 2-sums}\label{subsection:new from old}
Let~$M$ and~$N$ be matroids with $E(M)\cap E(N)=\{p\}$ where~$p$ is neither a
loop nor a bridge in either matroid. Define the \emph{parallel
connection}~$M\parallel_p N$ to be the matroid on the ground set
$E(M)\cup E(N)$ with circuits
\[
  \cir(M)\cup \cir(N)\cup \{(C\cup D)\setminus \{p\}:p\in C\in\cir(M)\text{ and
  } p\in D\in\cir(N)\}.
\]
Suppose~$M$ and $N$ have cycle systems $\cs$
and~$\csd$, respectively.  It is then straightforward to check
that~$M\!\parallel_{p}\!N$ has cycle system~$\cs\cup\csd$.  Define the
\emph{$2$-sum} of $M$ and~$N$ along~$p$ to be the matroid
$M\oplus_{2}N=(M\!\parallel_{p}\! N)\setminus \{p\}$.  See Figure~\ref{fig: 2-sum} for an example.

\begin{prop}\label{prop:two-sum}  With notation as above, suppose that~$M$ has cycle
  system $\cs=\{C_{1},\ldots,C_{g}\}$ and~$N$ has cycle
  system~$\csd=\{D_{1},\ldots,D_{f}\}$. Suppose that~$p\in C_{1}\cap
  D_{1}$ and $p\in\cs_{[g]}$. 
  Let~$\widetilde{\cs}=\{\tilde{C}_{2},\ldots,\tilde{C}_{g}\}$
  where~$\tilde{C}_{i}=C_{i}$ for $i=2,\ldots,g$, and let
  $\widetilde{\csd}=\{\tilde{D}_{1},\ldots,\tilde{D}_{f}\}$ where
  \[
    \widetilde{D}_{i} = 
    \begin{cases}
      \hfil D_{i}& \text{if~$p\not\in D_{i}$}\\
     (C_{1}\cup D_{i})\setminus \{p\}& \text{if~$p\in D_{i}$}.
    \end{cases}
  \]
  Then~$M\oplus_2N$ has cycle system $\mathcal{E} = \widetilde{\cs}\cup\widetilde{\csd}$. 
\end{prop}
\noindent\textsc{note:} The choice of~$p$ being in the first element of each
cycle system is only for  notational convenience.
\begin{proof}
  Let~$\sigma\subseteq [g]\setminus\{1\}$ and~$\tau\subseteq[f]$, and suppose
  that $\sigma$ or $\tau$ is nonempty. Let~$\mathcal{E}_{\sigma,\tau}$ denote
  the unique union of $\{\tilde{C}_{i}: i\in \sigma\}\cup\{\tilde{D}_{i}:i\in
  \tau\}$.  We must show~$\mathcal{E}_{\sigma,\tau}$ is dependent in~$M\oplus_2
  N$. 

  There are cases to consider. Let~$t_{p}=|\{i\in\tau: p\in D_{i}\}|$. First,
  suppose~$\sigma=\emptyset$.  In this case, $\tau\neq\emptyset$, and
  $\mathcal{E}_{\sigma,\tau}=\tilde{\mathcal{D}}_{\tau}=*\{\tilde{D}_{i}:i\in\tau\}$.
  Since~$\mathcal{D}$ is a cycle system on~$N$, there exists a circuit~$D$
  of~$N$ with~$D\subseteq \mathcal{D}_{\tau}$. If~$t_{p}\neq 1$,
  then~$\mathcal{E}_{\sigma,\tau}=\tilde{\mathcal{D}}_{\tau}=\mathcal{D}_{\tau}$,
  and ~$p\not\in D$. So $D$ is a circuit of~$M\oplus_{2}N$.  If~$t_{p}=1$, then
  $\mathcal{E}_{\sigma,\tau}=(C_{1}\cup\mathcal{D}_{\tau})\setminus\{p\}\supseteq
  (C_{1}\cup D)\setminus \{p\}$.  The circuit elimination property for
  $M\!\parallel_p\!N$ guarantees that~$(C_{1}\cup D)\setminus \{p\}$ contains a
  circuit of~$M\oplus_{2}N$.  Thus, in the case~$\sigma=\emptyset$, we 
  have that~$\mathcal{E}_{\sigma,\tau}$ is dependent.

  Next, suppose~$\sigma\neq\emptyset$. Since~$\cs$ is a cycle system,
  $\cs_{\sigma}$ contains a circuit~$C$ of~$M$.  If $t_{p}\neq1$,
  then~$\mathcal{E}_{\sigma,\tau}\supseteq \cs_{\sigma}$, and
  since~$1\not\in\sigma$, we have~$p\not\in C$. Hence, $C$ is a circuit
  of~$M\oplus_{2}N$,  and $\mathcal{E}_{\sigma,\tau}$ is dependent.
  If~$t_{p}=1$,
  then $\mathcal{E}_{\sigma,\tau}=(\cs_{\sigma\cup\{1\}}\cup\mathcal{D}_{\tau})\setminus\{p\}$.
  Let~$C'$ be a circuit in~$\cs_{\sigma\cup\{1\}}$, and let~$D$ be a
  circuit in~$\mathcal{D}_{\tau}$.  We
  have~$\mathcal{E}_{\sigma,\tau}\supseteq(C'\cup D)\setminus\{p\}$, and as
  above, the circuit elimination property guarantees the existence of a
  circuit of~$M\oplus_{2}N$ in $(C'\cup D)\setminus\{p\}$.
\end{proof}

\begin{cor}\label{cor:K33-free}
  If $G$ is a graph that is $K_{3,3}$-free, i.e., it does not contain~$K_{3,3}$ as a minor, then the cycle matroid
  for~$G$ has a cycle system.
\end{cor}
\begin{proof} If a graph~$G$ is not~$2$-connected, pick a vertex~$v$ whose
  removal disconnects the graph, and let~$G_{1},\ldots,G_{k}$ be the resulting
  connected components.  Then~$G$ is $K_{3,3}$-free if and only if each~$G_{i}$
  is $K_{3,3}$-free.  It is straightforward to check
  that~$M(G)=\bigoplus_{i=1}^{k}M(G_{i})$ and~$M(G)$ has a cycle system if and
  only if each~$M(G_{i})$ has a cycle system.  Hence, we may assume~$G$
  is~$2$-connected. 

  We use the fact that a $2$-connected graph is~$K_{3,3}$-free if and only if it can
  be constructed from planar graphs and copies of~$K_{5}$ via parallel connections
  and $2$-sums (see \cite{Khuller} and \cite{Vazirani} for details).  We must show
  the conditions of Proposition~\ref{prop:two-sum} apply.  Let~$H$ be either a
  planar graph or $K_{5}$, and let~$p\in H$.  We claim that~$H$ has a cycle
  system~$\cs$ such that~$p$ is an element of the unique
  union~$*\cs$.  If~$H$ is a planar graph, embed~$H$ in the plane so
  that~$p$ is adjacent to the unbounded face of~$H$, and take~$\cs$ to be
  the cycles corresponding to the bounded faces of~$H$. If~$H=K_{5}$, first note that $K_{5}$ has a cycle
  system since it is the cone over a planar graph.  Then our claim follows by the
  symmetry of the edges of~$K_{5}$.
\end{proof}

\begin{example}
Figure~\ref{fig: 2-sum} displays the $2$-sum of the complete graphs $K_4$ and~$K_5$ on vertex sets $\{1,2,6,7\}$ and  $\{2,3,4,5,6\}$, respectively.  For $K_4$, take the cycle system $\cs=\{126, 127, 167\}$ consisting of the triangles containing the vertex~$1$, and for $K_5$, take the cycle system $\mathcal{D}=\{423,\allowbreak 425,\allowbreak 426,\allowbreak 435,\allowbreak 436,\allowbreak 456\}$ consisting of the triangles containing vertex~$4$.  (These cycle systems arise from considering $K_4$ and $K_5$ as cones and applying the construction of Section~\ref{sec:examples}~(\ref{example:cone}).)  Letting $p$ be the edge~$26$, Proposition~\ref{prop:two-sum} yields the cycle system for our $2$-sum  which is the union of
\[
\widetilde{\cs}=\cs\setminus 126=
\{127,167\}
\]
and the cycles
derived from $\mathcal{D}$ by replacing the triangle~$426$ with the $4$-cycle~$4216$:
\[
\widetilde{\mathcal{D}}=\{423,425,4216,435,436,456\}.
\]
   
    \begin{figure}[ht]
    \centering
    \begin{tikzpicture}
       \foreach \i in {0,1,2,3,4} {
           \node[ball] (\i) at ({\i*72+144}:1) {};
       }
    \node[ball] (5) at (-1.9846,0.5877852) {};
    \node[ball] (6) at (-1.9846,-0.58778525) {};
    \node at (0) [above=0.4em] {$2$};
    \node at (1) [below=0.4em] {$6$};
    \node at (2) [below=0.4em] {$5$};
    \node at (3) [right=0.4em] {$4$};
    \node at (4) [above=0.4em] {$3$};
    \node at (5) [above=0.4em] {$1$};
    \node at (6) [below=0.4em] {$7$};
    \draw (0)--(2);
    \draw (0)--(3);
    \draw (0)--(4);
    \draw (0)--(6);
    \draw (1)--(2);
    \draw (1)--(3);
    \draw (1)--(4);
    \draw (1)--(5);
    \draw (2)--(3);
    \draw (2)--(4);
    \draw (3)--(4);
    \draw (0)--(5)--(6)--(1);
    \end{tikzpicture}
    \caption{2-sum of $K_4$ and $K_5$.}\label{fig: 2-sum}
    \end{figure}
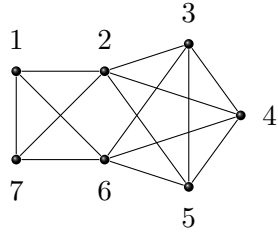
\end{example}

\section{The poset of coparking functions}\label{sec: coparking}

In this section we define our notion of a coparking function associated to a cycle system on a matroid $M$. We show that the set of such sequences forms a pure multicomplex whose degree sequence recovers the $h$-vector of $M$. We begin with our main definition.


\begin{defn}\label{defn:coparking}
Let $M$ be a matroid with cycle system $\cs = \{C_{1},\ldots,C_g\}$. A
\emph{coparking function} with respect to $\cs$ is a function
\begin{align*}
a:\cs\rightarrow \N\\
C_i\mapsto a_i
\end{align*}
with the property that for all nonempty $\sigma\subseteq[g]$, there exists some
$i\in\sigma$ such that $a_i<|C_i\cap\cs_\sigma|$.  
\end{defn}

We will always fix an
ordering $\cs=(C_1,\dots,C_g)$ and write coparking functions as vectors
$a=(a_1,\dots,a_g)\in\N^g$.  However, to emphasize that our results do
not fundamentally depend on this choice of  ordering, we will continue to write the
cycle system as a set~$\{C_{1},\ldots,C_{g}\}$.

The {\em degree} of a coparking function~$a$ is $\deg(a)=\sum_{i=1}^ga_i$.  The
set of coparking functions with respect to~$\cs$ is
denoted~$P^{*}(\cs)$ or simply $P^{*}$ if $\cs$ is
clear from context. The {\em coparking function degree sequence with respect
to~$\cs$} is 
  \[
    d(\cs):=(d_0,d_1,\dots)
  \]
  where~$d_i$ is the number of coparking functions having degree~$i$:
  \[
    d_i:=|\{a\in P^*: \deg(a)=i\}|.
  \]
  As a special case, if $\cs=\emptyset$, then $g=0$ and~$\N^{g}=\N^{0}=\{()\}=P^{*}$
  with~$\deg(())=0$.

  In general, define a partial order~$\le$ on~$P^*$ by~$a\le b$ if~$a_i\leq b_i$
  for~$i=1,\dots,g$. We write $a< b$ if $a\le b$ and $a\ne b$, and we say that $a$ is
  \emph{maximal} if there is no $b\in P^*$ such that $a< b$.

\subsection{Algorithm}
According to Definition \ref{defn:coparking}, determining whether a given integer sequence is a coparking function requires one to check all subsets of the $g$ elements in the underlying cycle system. It turns out that Dhar's algorithm \cite{Dhar90} on graphs can be adapted to provide a faster (polynomial in $g$) way to check this condition. To determine whether a non-negative integer vector~$a=(a_1,\dots,a_g)$ is a
coparking function with respect to~$\cs$, first set~$\sigma=[g]$. Next,
attempt to find~$i\in\sigma$ such that~$a_i<|C_i\cap \cs_{\sigma}|$.  If no
such~$i$ exists, then~$\sigma$ demonstrates that~$a$ is not a coparking
function.  Otherwise, replace~$\sigma$ by~$\sigma\setminus\left\{ i \right\}$.
(More generally, one may choose any nonempty subset~$\gamma$ of~$\sigma$ with
the property that~$a_i<|C_i\cap \cs_{\sigma}|$ for all $i\in\gamma$ and
replace~$\sigma$ by~$\sigma\setminus\gamma$.)  If~$\sigma=\emptyset$, then~$a$
is a coparking function.  Otherwise, repeat the previous steps with this
new~$\sigma$.

The algorithm is finite since at each round, either we have demonstrated that~$a$
is not a coparking function, or the cardinality of~$\sigma$ decreases.  To prove the
correctness of the algorithm, it remains to be shown that if the algorithm halts
with~$\sigma=\emptyset$, then~$a$ is a coparking function.  So
suppose~$\sigma=\emptyset$ when the algorithm halts, and let~$\tau$ be a nonempty
subset of~$[g]$.  As the algorithm runs, elements are removed from~$\sigma$
until a first element~$i\in\tau$ is removed.  Consider the set~$\sigma$ just
before~$i$ is removed.  We have~$a_i<|C_i\cap \cs_{\sigma}|$.  Then,
since~$i\in\tau\subseteq\sigma$, we have
\[
  |C_i\cap \cs_{\tau}|\geq|C_i\cap \cs_{\sigma}|>a_i.
\]
Thus,~$a$ is a coparking function. Pseudocode appears as
Algorithm~\ref{alg:coparking} below.

\RestyleAlgo{ruled}
\begin{algorithm}[H]\label{alg:coparking}
\caption{Verify coparking function}\label{alg:coPF}
\DataSty{input: cycle system $\cs=\{C_1,\dots,C_g\}$ and
$a=(a_1,\dots,a_g)\in \N^{g}$}\;
\DataSty{output: TRUE if~$a$ is a coparking function, FALSE if not}\;
$\sigma\gets[g]$\;
\While{$\sigma\neq \emptyset$}{
  \For{$i\in\sigma$}{
    \If{$a_i<|C_i\cap\cs_\sigma|$}{
      $\sigma=\sigma\setminus\{i\}$\;
      \KwSty{break}\tcc*{break out of the for-loop}
     }
  }\DataSty{output: FALSE}
}\DataSty{output: TRUE}
\end{algorithm}

\begin{example}\label{example: intro}
Consider the graph depicted below: 
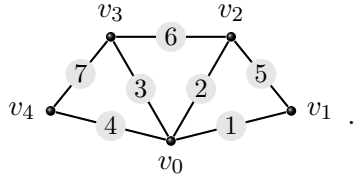
\begin{center}
  \begin{tikzpicture}[scale=0.8]
    \tikzset{mystyle/.style={midway,circle,minimum size=4mm,inner
    sep=0pt,fill=gray!20,thin,solid}};
    \node[ball,label={below:$v_0$}] (0) at (0,0) {};
    \node[ball,label={right:$v_1$}] (1) at (2,0.5) {};
    \node[ball,label={above:$v_2$}] (2) at (1,1.732) {};
    \node[ball,label={above:$v_3$}] (3) at (-1,1.732) {};
    \node[ball,label={left:$v_4$}] (4) at (-2,0.5) {};

    \draw[thick] (0)--(1) node[mystyle] {\small $1$};
   \draw[thick] (0)--(2) node[mystyle] {\small $2$};
    \draw[thick] (0)--(3) node[mystyle] {\small $3$};
    \draw[thick] (0)--(4) node[mystyle] {\small $4$};
   \draw[thick] (1)--(2) node[mystyle] {\small $5$};
   \draw[thick] (2)--(3) node[mystyle] {\small $6$};
    \draw[thick] (3)--(4) node[mystyle] {\small $7$};

    \node at (3,0.3) {.};
  \end{tikzpicture}
\end{center}
The three bounded faces define a circuit system $\cs=\{C_1,C_2,C_3\}$, with
\[
C_1=\{3,4,7\},\quad C_2=\{2,3,6\},\quad C_3=\{1,2,5\}.
\]
The Hasse diagram for the poset of coparking functions $P^*(\cs)$ appears in Figure~\ref{fig:poset of copf}.

\begin{figure}[ht]
\begin{center}
    \begin{tikzpicture}[
    scale=0.35,
    every node/.style={inner sep=1.5pt, font=\small}
]
\node (202) at (-9, 20) {(2,0,2)};
\node (211) at (-3, 20) {(2,1,1)};
\node (121) at (3, 20) {(1,2,1)};
\node (112) at (9, 20) {(1,1,2)};

\node (201) at (-12, 15) {(2,0,1)};
\node (210) at (-8, 15) {(2,1,0)};
\node (120) at (-4, 15) {(1,2,0)};
\node (111) at (0, 15) {(1,1,1)};
\node (021) at (4, 15) {(0,2,1)};
\node (102) at (8, 15) {(1,0,2)};
\node (012) at (12, 15) {(0,1,2)};

\node (200) at (-10, 10) {(2,0,0)};
\node (110) at (-6, 10) {(1,1,0)};
\node (020) at (-2, 10) {(0,2,0)};
\node (101) at (2, 10) {(1,0,1)};
\node (011) at (6, 10) {(0,1,1)};
\node (002) at (10, 10) {(0,0,2)};

\node (100) at (-4, 5) {(1,0,0)};
\node (010) at (0, 5) {(0,1,0)};
\node (001) at (4, 5) {(0,0,1)};

\node (000) at (0, 0.5) {(0,0,0)};

\draw (202) -- (201);
\draw (202) -- (102);
\draw (211) -- (201);
\draw (211) -- (210);
\draw (211) -- (111);
\draw (121) -- (120);
\draw (121) -- (021);
\draw (121) -- (111);
\draw (112) -- (102);
\draw (112) -- (012);
\draw (112) -- (111);

\draw (201) -- (200);
\draw (201) -- (101);
\draw (210) -- (200);
\draw (210) -- (110);
\draw (120) -- (110);
\draw (120) -- (020);
\draw (111) -- (110);
\draw (111) -- (101);
\draw (111) -- (011);
\draw (021) -- (020);
\draw (021) -- (011);
\draw (102) -- (101);
\draw (102) -- (002);
\draw (012) -- (011);
\draw (012) -- (002);

\draw (200) -- (100);
\draw (110) -- (100);
\draw (110) -- (010);
\draw (020) -- (010);
\draw (101) -- (100);
\draw (101) -- (001);
\draw (011) -- (010);
\draw (011) -- (001);
\draw (002) -- (001);

\draw (100) -- (000);
\draw (010) -- (000);
\draw (001) -- (000);
\end{tikzpicture}
\end{center}
\caption{Poset of coparking functions for the circuit system in Example~\ref{example: intro}.}\label{fig:poset of copf}
\end{figure}
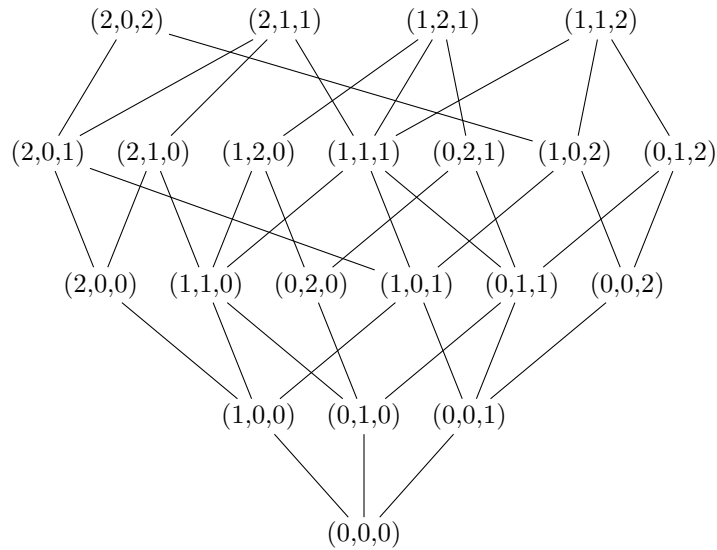

We illustrate Algorithm~\ref{alg:coparking} by verifying that $(2,0,2)$ is a coparking function, while $(2,2,0)$ is not. Running the algorithm on $(2,0,2)$ beginning with $\sigma=[3]$, we find that $a_2=0<|C_2\cap\cs_{[3]}|=|\{6\}|=1$. Updating $\sigma$ to $\sigma=\{1,3\}$, we now find that $a_1=2<|C_1\cap\cs_{\{1,3\}}|=|C_1|=3$. Updating $\sigma$ again to $\sigma=\{3\}$, we find that $a_3=2<|C_3\cap\cs_{\{3\}}|=3$, which leads to $\sigma=\emptyset$, and the conclusion that $(2,0,2)$ is coparking.

If we instead run the algorithm on $(2,2,0)$, we immediately find that $a_i\ge |C_i\cap\cs_{[3]}|$ for $i=1,2,3$. Hence, the algorithm reports that $(2,2,0)$ is not a coparking function.

The degree sequence is $d(\cs)=(1,3,6,7,4)$, with four maximal coparking functions, each of degree $4$. In particular, we see that $P^*(\cs)$ is a pure multicomplex. The next proposition shows that this holds true in general.
\end{example}

\subsection{Coparking functions form a pure multicomplex}

\begin{prop}\label{prop:pure} Let $M$ be a matroid with cycle system
  $\cs = \{C_{1},\ldots,C_g\}$. Then the poset of coparking functions
  $P^*(\cs)$ is a pure multicomplex.  An element~$a\in P^*(\cs)$
  is maximal if and only if $\deg(a)=\rk(M)-\beta$, where $\beta$ is the number of bridges
  of~$M$.
\end{prop}
\begin{proof} The fact that~$P^{*}(\cs)$ is an order ideal follows
  immediately from the definition of a coparking function.  Now let~$a\in
  P^*(\cs)$.  Run Algorithm~\ref{alg:coparking} to produce a listing
  $i_1,\dots,i_g$ of the elements of~$[g]$ such that
\[
  a_{i_k}< |U_k|
\]
where
\[
  U_k:=C_{i_k}\cap \cs_{\{i_{k},\dots,i_g\}}
\]
for~$k=1,\dots,g$. Let~$\Pi$ be the set of bridges of~$M$. We saw in
Corollary~\ref{cor:non-bridges} that~$E\setminus \Pi=\bigcup_{i=1}^{g}C_{i}$.
Given~$e\in E\setminus\Pi$, let~$j=\max\{k:e\in C_{i_k}\}$.  Then~$e\in U_{j}$ and in no
other~$U_{k}$.  Hence,~$E\setminus \Pi$ is the disjoint union of the~$U_k$:
\[
  \bigsqcup_{k=1}^gU_k=E\setminus \Pi.
\]
Therefore,
\[
  \deg a=\sum_{k=1}^ga_{i_k}\leq\sum_{k=1}^g(|U_k|-1)=|E\setminus
  \Pi|-g=|E|-|\Pi|-g=\rk(M)-\beta.
\]
Define $c=(c_1,\dots,c_g)$ where $c_{i_k}:=|U_k|-1$ for~$k=1,\dots,g$.  Then the
same sequence of steps we used in the algorithm to show~$a$ is a coparking
function now shows that~$c$ is a coparking function.  Further,~$a\le c$,
and~$\deg(c)=\rk(M)-\beta$. It follows that~$P^{*}(\cs)$ is pure with maximal elements
having degree~$\rk(M)-\beta$.
\end{proof}

\subsection{Deletion and Contraction}

\begin{prop}\label{prop:del-con coparking}
  Let~$\cs = \{C_{1},\ldots,C_g\}$ be a cycle system for the matroid~$M$.
  Let~$e\in\cs_{[g]}\cap C_{i}$ for some~$i$, and let~$\cs'$
  and~$\cs''$ be cycle systems for~$M\setminus e$ and~$M/e$,
  respectively, as in Proposition~\ref{prop:del-con cs}
  (where~$\cs''$ is only defined if~$e$ is a non-loop).

  \begin{enumerate}
    \item\label{item:del-con coparking 1} There is an injective mapping of
      coparking functions $P^{*}(\cs')\hookrightarrow
      P^{*}(\cs)$ given by
      \[
        (a_{1},\ldots,a_{i-1},a_{i+1},\ldots,a_{g})\mapsto
	(a_{1},\ldots,a_{i-1},0,a_{i+1},\ldots,a_{g}).
      \]
      The image is the set $\{(a_{1},\ldots,a_{g})\in P^{*}(\cs):
      a_{i}=0\}$.  If~$e$ is a loop, then this mapping is bijective.
    \item\label{item:del-con coparking 2} Suppose~$e$ is a non-loop.  Then
      there is an injective mapping of
      coparking functions $P^{*}(\cs'')\hookrightarrow P^{*}(\cs)$ given by
      \[
        (a_{1},\ldots,a_{g})\mapsto(a_{1},\ldots,a_{i-1},a_{i}+1,a_{i+1},\ldots,a_{g}).
      \]
      The image is the set~$\{(a_{1},\ldots,a_{g})\in P^{*}(\cs):
      a_{i}>0\}$.
  \end{enumerate}
\end{prop}
\begin{proof}
  Without loss of generality, let~$i=g$.
  \smallskip

  \noindent(\ref{item:del-con coparking 1}) Let~$a'=(a_{1},\ldots,a_{g-1})\in
  \N^{g-1}$ and~$a=(a_{1},\ldots,a_{g-1},0)$.  We must show~$a'\in
  P^{*}(\cs')$ if and only if~$a\in P^{*}(\cs)$.  Let~$\emptyset\neq\sigma\subseteq[g]$.
  If~$g\in\sigma$, then $a_{g}=0<1\leq|\cs_{\sigma}\cap C_{g}|$, and
  if~$g\not\in\sigma$, then
  $\cs'_{\sigma}=\cs_{\sigma}$.  The result follows.
  If $e$ is a loop, then Corollary~\ref{cor:loops} implies $c_{g}=0$ for
  every coparking function~$c\in P^{*}(\cs)$.  So in this case, our mapping is
  bijective.
  \smallskip

  \noindent(\ref{item:del-con coparking 2}) Let~$a'=(a_{1},\ldots,a_{g})\in
  \N^{g}$ and~$a=(a_{1},\ldots,a_{g-1},a_{g}+1)$. We must show~$a'\in
  P^{*}(\cs'')$ if and only if~$a\in P^{*}(\cs)$.  We
  have~$\cs''=\{C_{1}'',\ldots,C_{g}''\}$ where
  \[
    C_{j}''=
    \begin{cases}
     \hfil C_{j}&\text{if~$j\neq g$,}\\
      C_{j}\setminus\{e\}&\text{if $j=g$.}
    \end{cases}
  \]
  Therefore, for all nonempty~$\sigma\subseteq[g]$,
  \[
    \cs_{\sigma}=
    \begin{cases}
      \hfil \cs''_{\sigma}&\text{if~$g\not\in\sigma$,}\\
      \cs''_{\sigma}\cup\{e\}&\text{if $g\in\sigma$,}
    \end{cases}
  \]
  from which it follows that
  \[
    |\cs_{\sigma}\cap C_{j}|=
    \begin{cases}
      \hfil |\cs''_{\sigma}\cap C''_{j}|&\text{if~~$j\neq g$ or
      $g\not\in\sigma$,}\\
      |\cs''_{\sigma}\cap C''_{j}|+1&\text{if $j=g$ and $g\in\sigma$.}
    \end{cases}
  \]
  The result follows.
\end{proof}

\subsection{Main result} 
We have now seen that cycle systems behave well with respect to deletion and contraction, with similar properties carrying over to the corresponding coparking functions. From these observations we obtain our main result.

\begin{thm}\label{thm:main}
  Let~$\cs=\{C_{1},\ldots,C_{g}\}$ be a cycle system for a matroid~$M$,
  and let~$P^{*}=P^{*}(\cs)$ be the corresponding set of coparking functions.  Then the
  degree vector of~$P^{*}$ is the~$h$-vector for~$M$, i.e., $\deg(P^{*})=h(M)$.
\end{thm}
\begin{proof} We proceed by induction on the number of non-bridges of~$M$.
  Suppose $M$ consists solely of bridges. Then~$\cs=\emptyset$
  and~$g=0$.  We have $\N^{g}=\N^{0}=\{()\}$ with~$\deg(())=0$.  The Tutte
  polynomial of~$M$ is $T_{M}(x,y)=x^{\rk(M)}$.
  Hence,~$\deg(P^{*})=h(M)=(1)$.

  Now suppose~$M$ contains a non-bridge. Then $g>0$ and we can take $e\in
  \cs_{[g]}\cap C_{i}$ for some~$i$.
  Let~$\cs'=\cs\setminus C_{i}$ be the cycle system
  on~$M\setminus e$ from Proposition~\ref{prop:del-con coparking}. We first consider
  the case where~$e$ is a loop. In this case~$T_{M}(x,y)=yT_{M\setminus
  e}(x,y)$, and hence, $T_{M}(x,1)=T_{M\setminus e}(x,1)$.  So~$M$ and~$M
  \setminus e$ have the same~$h$-vectors.  On the other hand, by
  Proposition~\ref{prop:del-con coparking}, there is a degree-preserving
  bijection~$P^{*}(\cs')\to P^{*}(\cs)$. So in this case, our
  result follows by induction.

  We now consider the case where~$e$ is neither a loop nor a bridge.
  Let~$\cs'$ and $\cs''$ be the cycle systems from
  Proposition~\ref{prop:del-con coparking} on $M\setminus e$ and $M/e$,
  respectively.  We have~$T_{M}(x,y)=T_{M\setminus e}(x,y)+T_{M/e}(x,y)$, and
  hence,~$h(M)=h(M\setminus e)+h(M/e)$. Then,
  by Proposition~\ref{prop:del-con coparking} and induction,
  \[
    \deg(P^{*}(\cs)) 
    = \deg(P^{*}(\cs'))+ \deg(P^{*}(\cs''))
    =h(M\setminus e)+h(M/e)
    =h(M),
  \]
  as required.

  We now consider the case where~$e$ is neither a loop nor a bridge.
  If $a=(a_0,\ldots,a_k)$ is an integer vector, let $a^{+}$ denote the vector obtained from $a$ by prepending a zero: $a^{+}:=(0,a_0,\ldots,a_k)$.
  We have~$T_{M}(x,y)=T_{M\setminus e}(x,y)+T_{M/e}(x,y)$, and since $\rk(M)=\rk(M\setminus e)=\rk(M/e)+1$, it follows that
  $h(M)=h(M\setminus e)+h(M/e)^{+}$.
  Let~$\cs'$ and $\cs''$ be the cycle systems from
  Proposition~\ref{prop:del-con coparking} on $M\setminus e$ and $M/e$,
  respectively.   Then,
  by Proposition~\ref{prop:del-con coparking} and induction,
  \[
    \deg(P^{*}(\cs)) 
    = \deg(P^{*}(\cs'))+ \deg(P^{*}(\cs''))^{+}
    =h(M\setminus e)+h(M/e)^{+}
    =h(M),
  \]
  as required.  The vector $\deg(P^{*}(\cs''))^{+}$ appears since the injection $P^*(\cs'')\hookrightarrow P^*(\cs)$ shifts degrees by one.
\end{proof}

\begin{cor}\label{cor:Stanley}
    Conjecture \ref{conj:Stanley} holds for any matroid that admits a cycle system.
\end{cor}

\section{Bijections between bases and coparking functions}\label{sec:bijections}
A \emph{rooted tree} is a tree~$T$ with a distinguished vertex called the
\emph{root}. The \emph{level} of a vertex of~$T$ is the number of edges in the
unique path from that vertex to the root.  If $v,w$ are adjacent vertices of~$T$
and the level of~$w$ is one greater than that of $v$, then $w$ is the
\emph{child} of~$v$ and $v$ is the \emph{parent} of~$w$.  The rooted tree~$T$ is
a \emph{full binary tree} if each vertex has either no children or exactly two
children, one of which is designated as the \emph{left child} and the other as
the \emph{right child}.  The edge from a vertex to a left (resp., right) child
is called a \emph{left} (resp.,~\emph{right}) \emph{edge}.  A full binary tree
can be described recursively as either (i) a singleton set (the root), or (ii) a
tuple~$(L,r,R)$ where $r$ is a singleton set (the \emph{root}) and both $L$ and
$R$ are full binary trees.

\begin{defn} Let $M$ be a matroid. A \emph{deletion/contraction-tree} or
  \emph{DC-tree} for a
  matroid~$M$, denoted $T(M)$, is a full binary tree constructed
  recursively as follows: (i)~$T(M)=M$ if $M$ consists solely of loops and
  bridges, and otherwise, (ii) $T(M)=(T(M\setminus e),M,T(M/e))$ for some $e\in
  E(M)$ that is neither a bridge nor a loop.  
\end{defn}

\subsection{DC-tree determined by a cycle system and ground set
ordering}  Let~$M$ be a matroid with cycle system $\cs =
\{C_{1},\ldots,C_g\}$.  Then a linear ordering of the ground set~$\xi$ of~$M$
determines a DC-tree with labeled right edges, denoted~$T(\cs,\xi)$,
which we now describe via a recursive algorithm: 

\begin{enumerate}
  \item\label{step1} If $M$ consists solely of loops and bridges, then
    $T(\cs,\xi)=M$, and we are done. Otherwise, proceed to the next
    step.
  \item Set $\sigma=[g]$.
  \item\label{step3} Let~$e$ be the maximal element of~$\cs_\sigma$, and
    let $i$ be the unique index such that~$e\in C_{i}$. 
  \item\label{step4} If~$e$ is a loop, replace~$\sigma$ with $\sigma\setminus\{i\}$ and repeat Step~(\ref{step3}).
  \item\label{step5} Let~$\cs'$ and $\cs''$ be the cycle systems on $M\setminus
  e$ and~$M/e$ defined in Proposition~\ref{prop:del-con cs}.  Remove~$e$
  from~$\xi$ but otherwise retain the linear order.  Recursively construct~$T(\cs',\xi)$
  and~$T(\cs'',\xi)$, and assign their roots as the left and right children, respectively, of $M$ to form the full binary tree $(T(\cs',\xi),M,T(\cs'',\xi))$.
\item Label the edge from $M$ to the root of $T(\cs'',\xi)$ by the
  pair~$(e,C_i)$.
\end{enumerate}

\begin{remark}
  In Step~(\ref{step4}) the set~$\sigma$ never becomes empty. To see this,
  suppose~$\sigma$ did become empty. Then we would
  have found $g$ loops in~$M$.  Since $\rk(M)=|E(M)|-g$, the elements of~$M$
  besides these~$g$ loops form a basis for~$M$, and this would be the only basis
  for~$M$.  Thus,~$M$ consists solely of loops and bridges, and
  our algorithm would have stopped at Step~(\ref{step1}).  Also, note that the
  recursion is finite: the size of $\xi$ decreases by one each time
  Step~(\ref{step5}) is reached in the recursion, and when $|\xi|=1$ the matroid
  in question consists of a single loop or a single bridge.
\end{remark}

\subsection{Basis-to-coparking bijections}\label{subsect:basis-to-cpkg}
Our goal now is to show that the choice of a cycle system and an ordering of the
ground set determines a bijection between the set of bases of the matroid
and the set of coparking functions with respect to that cycle system. 
The reader may find it helpful to consult Figure~\ref{fig:dc-diagram} while reading this section.
\begin{figure}[htpb]
  \centering
  \begin{tikzpicture}[scale=0.8]

    \begin{scope}[local bounding box=root]
      \node[ball] (v1) at (0.75,1.3) {}; 
      \node[ball] (v2) at (0,0) {};
      \node[ball] (v3) at (1.5,0) {}; 

      \draw (v1) -- (v2);  
      \draw (v1) -- (v3);  
      \draw (v2) -- (v3) node[midway, inner sep=2pt, fill=white] {$c$};  
      \node at (0.72,-1.2) (cp) {};  
      \draw (v2) .. controls (cp) .. (v3) node[midway, below] {$d$};

      \node[fill=white, inner sep=1pt] at ($(v1)!0.5!(v2)!0.2cm!90:(v1)$) {$a$}; 
      \node[fill=white, inner sep=1pt] at ($(v1)!0.5!(v3)!0.2cm!270:(v1)$) {$b$}; 

      \node at (barycentric cs:v1=1,v2=1,v3=1) {$\mathbbm{1}$}; 
      \node at (barycentric cs:v2=1,cp=1,v3=1) {$\mathbbm{2}$}; 
    \end{scope}
    \node[fit=(root), inner sep=0.2cm, line width=1pt] (padded_root) {}; 

    \begin{scope}[local bounding box=l, shift={($(root.west)+(-5cm,-3.5cm)$)}]
      \node[ball] (v1) at (0.75,1.3) {}; 
      \node[ball] (v2) at (0,0) {};
      \node[ball] (v3) at (1.5,0) {}; 

      \draw (v1) -- (v2);  
      \draw (v1) -- (v3);  
      \draw (v2) -- (v3) node[midway, below, inner sep=2pt, fill=white] {$c$};  

      \node[fill=white, inner sep=1pt] at ($(v1)!0.5!(v2)!0.3cm!90:(v1)$) {$a$}; 
      \node[fill=white, inner sep=1pt] at ($(v1)!0.5!(v3)!0.3cm!270:(v1)$) {$b$}; 

      \node at (barycentric cs:v1=1,v2=1,v3=1) {$\mathbbm{1}$}; 
    \end{scope}

    \begin{scope}[local bounding box=r, shift={($(root.east)+(5cm,-3cm)$)}]
      \node[ball] (v1) at (0,1.3) {}; 
      \node[ball] (v2) at (0,0) {};

      \draw[bend right = 60] (v1) to node[midway, left, inner sep=2pt, fill=white] {$a$} (v2);
      \draw[bend left = 60] (v1) to node[midway, right, inner sep=2pt, fill=white] {$b$} (v2);
      \draw (v2) to[loop below,out=270-40,in=270+40,looseness=1.5,min distance=1.7cm]
      node[pos=0.5, below, inner sep=2pt, fill=white] {$c$}
      coordinate[pos=0.5] (loop_bottom) (v2);

      \node at ($(v1)!0.5!(v2)$) {$\mathbbm{1}$};
      \node at ($(loop_bottom) + (0,0.45cm)$) {$\mathbbm{2}$};
    \end{scope}

    \begin{scope}[local bounding box=ll, shift={($(l.center)+(-2.3cm,-4.5cm)$)}]
      \node[ball] (v1) at (0.75,1.3) {}; 
      \node[ball] (v2) at (0,0) {};
      \node[ball] (v3) at (1.5,0) {}; 

      \draw (v1) -- (v2);  
      \draw (v1) -- (v3);  

      \node[fill=white, inner sep=1pt] at ($(v1)!0.5!(v2)!0.3cm!90:(v1)$) {$a$}; 
      \node[fill=white, inner sep=1pt] at ($(v1)!0.5!(v3)!0.3cm!270:(v1)$) {$b$}; 
    \end{scope}

    \begin{scope}[local bounding box=lr, shift={($(l.center)+(2.3cm,-4.5cm)$)}]
      \node[ball] (v1) at (0,1.3) {}; 
      \node[ball] (v2) at (0,0) {};

      \draw[bend right = 60] (v1) to node[midway, left, inner sep=2pt, fill=white] {$a$} (v2);
      \draw[bend left = 60] (v1) to node[midway, right, inner sep=2pt, fill=white] {$b$} (v2);

      \node at ($(v1)!0.5!(v2)$) {$\mathbbm{1}$};
    \end{scope}

    \begin{scope}[local bounding box=rl, shift={($(r.center)+(-2.5cm,-3.5cm)$)}]
      \node[ball] (v1) at (0,1) {}; 
      \node[ball] (v2) at (0,0) {};

      \draw (v1) to node[midway, left, inner sep=2pt, fill=white] {$a$} (v2);
      \draw (v2) to[loop below,out=270-40,in=270+40,looseness=1.5,min distance=1.7cm]
      node[pos=0.5, below, inner sep=2pt, fill=white] {$c$}
      coordinate[pos=0.5] (loop_bottom) (v2);

      \node at ($(loop_bottom) + (0,0.45cm)$) {$\mathbbm{2}$};
    \end{scope}

    \begin{scope}[local bounding box=rr, shift={($(r.center)+(2.5cm,-3.5cm)$)}]
      \node[ball] (v1) at (0,0) {};

      \draw (v1) to[loop above,out=90-40,in=90+40,looseness=1.5,min distance=1.7cm]
      node[pos=0.5, above, inner sep=2pt, fill=white] {$a$}
      coordinate[pos=0.5] (loop_top) (v1);
      \draw (v1) to[loop below,out=270-40,in=270+40,looseness=1.5,min distance=1.7cm]
      node[pos=0.5, below, inner sep=2pt, fill=white] {$c$}
      coordinate[pos=0.5] (loop_bottom) (v1);

      \node at ($(loop_top) - (0,0.45cm)$) {$\mathbbm{1}$};
      \node at ($(loop_bottom) + (0,0.45cm)$) {$\mathbbm{2}$};
    \end{scope}

    \begin{scope}[local bounding box=lrl, shift={($(lr.center)+(-1.5cm,-4cm)$)}]
      \node[ball] (v1) at (0,1) {}; 
      \node[ball] (v2) at (0,0) {};

      \draw (v1) to node[midway, left, inner sep=2pt, fill=white] {$a$} (v2);

    \end{scope}

    \begin{scope}[local bounding box=lrr, shift={($(lr.center)+(1.5cm,-4cm)$)}]
      \node[ball] (v1) at (0,0) {};

      \draw (v1) to[loop above,out=90-40,in=90+40,looseness=1.5,min distance=1.7cm]
      node[pos=0.5, above, inner sep=2pt, fill=white] {$a$}
      coordinate[pos=0.5] (loop_top) (v1);

      \node at ($(loop_top) - (0,0.45cm)$) {$\mathbbm{1}$};
    \end{scope}

    \draw[shorten <=1cm, shorten >=1cm] (root.center)--(l.center);
    \node[fill=white, inner sep=1pt] at ($(root.center)!0.5!(r.center)!0.6cm!90:(r.center)$) {$(d,\mathbbm{2})$}; 
    \draw[shorten <=1cm, shorten >=1cm] (root.center)--(r.center);
    \draw[shorten <=1cm, shorten >=1cm] (l.center)--(ll.center);
    \node[fill=white, inner sep=1pt] at ($(l.center)!0.5!(lr.center)!0.65cm!90:(lr.center)$) {$(c,\mathbbm{1})$}; 
    \draw[shorten <=1cm, shorten >=1cm] (l.center)--(lr.center);
    \draw[shorten <=1cm, shorten >=1cm] (r.center)--(rl.center);
    \node[fill=white, inner sep=1pt] at ($(r.center)!0.5!(rr.center)!0.65cm!90:(rr.center)$) {$(b,\mathbbm{1})$}; 
    \draw[shorten <=1cm, shorten >=1cm] (r.center)--(rr.center);
    \draw[shorten <=0.8cm, shorten >=0.8cm] (lr.center)--(lrr.center);
    \node[fill=white, inner sep=1pt] at ($(lr.center)!0.5!(lrr.center)!0.68cm!90:(lrr.center)$) {$(b,\mathbbm{1})$}; 
    \draw[shorten <=0.8cm, shorten >=0.8cm] (lr.center)--(lrl.center);

    \node[below=0.2cm of ll.south, draw, align=left] {$(0,0)$\\$\{a,b\}$}; 
    \node[below=0.2cm of lrl.south, draw,align=left] {$(1,0)$\\$\{a,c\}$}; 
    \node[below=0.2cm of lrr.south, draw,align=left] {$(2,0)$\\$\{b,c\}$}; 
    \node[below=0.2cm of rl.south, draw,align=left] {$(0,1)$\\$\{a,d\}$}; 
    \node[below=0.2cm of rr.south, draw,align=left] {$(1,1)$\\$\{b,d\}$}; 
  \end{tikzpicture}
  \caption{The DC-tree for the cycle matroid of a graph~$G$ with cycle system
  $\cs=\{\mathbbm{1},\mathbbm{2}\}$ where $\mathbbm{1}=\{a,b,c\}$  and
$\mathbbm{2}=\{c,d\}$ and with ground set ordering~$a<b<c<d$.  Each leaf of the
tree is labeled with its corresponding coparking function for~$\cs$ and
spanning tree of~$G$, as described in Section~\ref{subsect:basis-to-cpkg}. The first and second components of the coparking function count the number of times $\mathbbm{1}$ and~$\mathbbm{2}$, respectively, appear as labels on the path from the leaf to the root. }
  \label{fig:dc-diagram}
\end{figure}
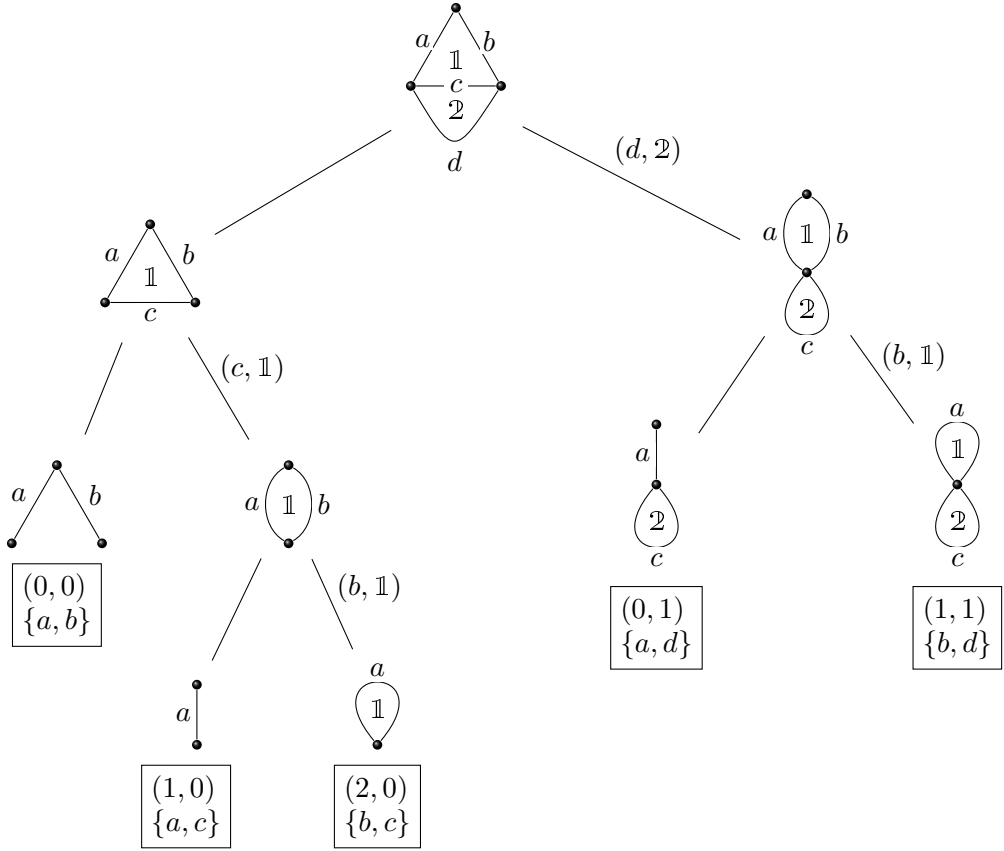
Let~$M$ be a matroid with non-loop, non-bridge element~$e$.  The bases
of~$M\setminus e$ are exactly the bases of~$M$ that do not contain~$e$, and the
bases of~$M/e$ are in bijection with the bases of~$M$ that contain~$e$ via
$B\mapsto B\cup e$ for each basis~$B$ of~$M/e$. Similarly, if~$\cs =
\{C_{1},\ldots,C_g\}$ is a cycle system for $M$ and $e\in \cs_{[g]}$, then
Proposition~\ref{prop:del-con cs} shows that~$M\setminus e$ and $M/e$ inherit
cycle systems, and Proposition~\ref{prop:del-con coparking} says the coparking
functions for those cycle systems correspond with the coparking
functions for~$\cs$.  Apply these correspondences recursively to a
DC-tree~$T(\cs,\xi)$, constructed as above. Each leaf
of~$T(\cs,\xi)$ is a minor~$N$ of~$M$ consisting solely of loops and
bridges.  The minor $N$ has a unique basis, which corresponds to a basis~$B$
of~$M$. The basis~$B$ consists of the bridge elements of~$N$ and all of the
edges that were contracted to obtain~$N$ from~$M$ in~$T(\cs,\xi)$.
These contracted edges will appear as part of the labels for the right
(contraction) edges along the unique path from~$N$ to~$M$
in~$T(\cs,\xi)$.   Similarly, by Proposition~\ref{prop:pure},~$N$ has a
unique coparking function (all of whose components are~$0$), and this coparking
function corresponds to a coparking function~$a$ on~$M$ with respect
to~$\cs$.  The~$i$-th component of~$a$ is the number of times a label of
the form $(e,C_i)$ for some~$e$ appears as a right (contraction) edge along the
unique path from~$N$ to~$M$ in $T(\cs,\xi)$.  In this way, we may label
each leaf~$N$ with a basis~$B(N)$ of~$M$ and a coparking function~$a(N)$
for~$\cs$.  Each basis and coparking function will appear exactly once
among these labels, and together, they account for all the bases and coparking
functions for the pair~$(M,\cs)$.  The mapping~$B(N)\mapsto a(N)$ as $N$
varies among the leaves of~$T(\cs,\xi)$ is thus a bijection between
bases and coparking functions for~$(M,\cs)$. We formalize this bijection
in Algorithms~\ref{alg:basis-to-coparking} and~\ref{alg:coparking-to-basis} below.

\RestyleAlgo{ruled}
\begin{algorithm}\label{alg:basis-to-coparking}
\caption{Basis-to-coparking}
\DataSty{input: cycle system $\cs=\{C_1,\dots,C_g\}$ on a matroid~$M$; a
linear ordering of the ground set of~$M$; and a basis~$B$ of~$M$}\;
\DataSty{output: coparking function~$a\in \N^{g}$ with respect to $\cs$}\;
$\sigma\gets[g]$\;
$a\gets \vec{0}\in \N^{g}$\;
\While{$\sigma\neq \emptyset$}{
  $e\gets \max(\cs_{\sigma})$\;
  $i\gets j$ where $j\in\sigma$ and $e\in C_{j}$\;
  \If{$e\notin B$}
  {$\sigma\gets\sigma\setminus\{i\}$\;}
  \Else
  {
    $a_{i}\gets a_{i}+1$\;
    $C_{i}\gets C_{i}\setminus \{e\}$\;
  }
}
  \DataSty{output: $a$}
\end{algorithm}

\RestyleAlgo{ruled}
\begin{algorithm}\label{alg:coparking-to-basis}
\caption{Coparking-to-basis}
\DataSty{input: cycle system $\cs=\{C_1,\dots,C_g\}$ on a matroid~$M$; a
linear ordering of the ground set~$E$ of~$M$; and a coparking function $a\in
\N^{g}$}\;
\DataSty{output: basis for~$M$}\;
$\sigma\gets[g]$\;
$B\gets E$\;
\While{$\sigma\neq \emptyset$}{
  $e\gets \max(\cs_{\sigma})$\;
  $i\gets j$ where $j\in\sigma$ and $e\in C_{j}$\;
  $a_{i}\gets a_{i}-1$\;
  \If{$a_{i}<0$}
  {$\sigma\gets\sigma\setminus\{i\}$\;
  $B\gets B\setminus\{e\}$\;}
  \Else
  {
    $C_{i}\gets C_{i}\setminus \{e\}$\;
  }
}
  \DataSty{output: $B$}
\end{algorithm}

\FloatBarrier

\bigskip
\section{Open questions}\label{sec:questions}

We end with some open questions and possible directions for future research.  
\subsection{Existence of cycle systems}
In Section~\ref{subsection:new from old} we give methods for forming new matroids with cycle systems from old ones via a gluing process, and we use those results to show the $K_{3,3}$-free graphs have cycle systems.  Can we perhaps extend these results?  Note that the collection of matroids with cycle systems is not closed under taking minors, and hence, even for the collection of graphic matroids, it is not clear whether there is a characterizing condition involving excluded minors.

\begin{prob} 
Describe all matroids with cycle systems.
\end{prob}

In our search for circuit systems, we have found no examples of nonbinary matroids that admit such structures. In addition , we have seen in Proposition \ref{prop:circuitspace} that elements of a cycle system have connections to the underlying circuit space, which can be used to characterize binary matroids. Inspired by this, we propose the following.
\begin{conj}\label{conj: nonbinary}
All matroids with cycle systems are binary.
\end{conj}

\noindent
As discussed above, all known examples of matroids that admit cycle systems are, in fact, regular.

Generalizing the question of whether a circuit system exists for a given matroid, we can try to count the number of cycle systems.  For instance, the graph in Figure~\ref{fig: 7 no fundamental} has a unique cycle system.  The complete graph $K_n$ may be thought of as a cone over $K_{n-1}$ in $n$ ways (one for each choice of cone vertex), and thus has the cone cycle system described in Section~\ref{sec:examples}.

\begin{question}  For $n\geq 5$, is it true that $K_n$ has only these cycle systems, and thus has only one cycle system up to symmetry?
\end{question}

\begin{question}
If a matroid $M$ admits a cycle system, when it is unique (up to symmetries of $M$)?
\end{question}

\begin{conj}
Let $G_n$ be the connected graph with $n$ vertices obtained by gluing $n-2$ triangles along a single common edge.  We conjecture that the number of cycle systems for $G_n$ is $(n-1)^{(n-3)}$, and that among all simple graphs with $n$ vertices, $G_n$ has the maximal number of cycle systems. 
\end{conj}

\begin{prob}
Are there classes of matroids for which the number of cycle systems is combinatorially interesting?  
\end{prob}

\subsection{Generalizing cycle systems.}  
Consider the recursive construction of a DC-tree for a matroid~$M$.  At each
stage, we consider a node of the tree  labeled by a minor $N$ of $M$, and $N$
will have a cycle system $\mathcal{D}$ derived from its parent.  As long as $N$
does not consist solely of bridges and loops, there will exist an element
$e\in\ast\mathcal{D}$.  We then form new nodes $N\setminus e$ and
$N/e$ for the tree with cycle systems as described in Proposition~\ref{prop:del-con cs}.  Thus, to continue the construction, we only need that the unique union $\ast\mathcal{D}$ be \emph{nonempty}, not that it contains a circuit.  Therefore, it is possible that a set of $g(M)$ cycles that does not form a cycle system might be used to construct a DC-tree $T$ for~$M$.  Labeling the leaves of $T$ as described in Section~\ref{subsect:basis-to-cpkg}, produces a collection of integer vectors~$P^*$.  One may then attempt to show that $P^*$ is a pure multicomplex with degree vector equal to the $h$-vector of $M$.

As an example, suppose $G = K_{3,3}$ is the complete bipartite graph on vertex set $\{0,1,2\} \sqcup \{3,4,5\}$. As we have noted, $G$ does not admit a cycle system. Consider the set of $4$-cycles  $C_1 = \{03, 05, 23, 25\}$, $C_2 = \{03,05,13, 15\}$, $C_3 = \{03, 04, 13, 14\}$,  and $C_4=\{03, 04, 23, 24\}$, and define $\cs = \{C_1, C_2, C_3, C_4\}$. One can check that labeling the DC-tree corresponding to these cycles using an edge-ordering
\[
(2, 4)< (0, 4)< (1, 5)< (0, 3)< (1, 4)< (2, 3)< (0, 5)< (2, 5)< (1, 3)
\]
to make choices at nodes, yields the  pure multicomplex with $20$ maximal elements formed by all cyclic permutations of the vectors
\[
(0,2,1,2), (0,2,0,3), (0,1,1,3), (0,1,3,1), (0,3,1,1).
\]
The degree sequence of the multicomplex, $(1, 4, 10, 20, 26, 20)$, is the $h$-vector of $K_{3,3}$.  In this case, one can check that, for any nonempty $\sigma \subseteq [4]$, the unique union $\cs_{\sigma}$ either contains a circuit, or else is a spanning tree for the graph induced on the (usual) union of elements $\cup_{i \in \sigma} C_i$. For example, we have $\cs_{\{1,2,3\}} = \{04, 23, 25, 14, 15\}$.

\subsection{Activity preserving bijections}

Our next open question involves the \emph{activity} of a basis of a matroid $M$. To recall this notion, fix a linear ordering $<$ on the ground set $E$ of $M$. For a basis~$B$, we say that $e \in B$ is \emph{internally passive} in $B$ if it can be replaced by a smaller element to obtain another basis; that is, if $(B \setminus e) \cup e^\prime$ is a basis of ${\mathcal M}$ for some $e^\prime < e$. We say that $e \in B$ is \emph{internally active} if it is not internally passive. An element $e \notin B$ is said to be \emph{externally passive (active)} if it is not (is) the smallest element in the unique circuit containing $B \cup e$. 

It turns out that activity can be used to recover/define the Tutte polynomial of a matroid $M$. In particular, one can show that 
\[T_{M}(x,y) = \sum \tau_{i,j} x^iy^j,\]
\noindent
where $\tau_{i,j}$ is the number of bases of $M$ with $i$ internally active elements and $j$ externally active elements. Note that $\tau_{i,j}$ is independent of the choice of ordering.
We refer to \cite{BrylawskiOxley1992} for more details regarding the Tutte polynomial and external activity.  From this, we can recover the $h$-vector  $(h_0, h_1, \dots, h_d)$ of $M$ in terms of activity: $h_i$ equals the number of bases with $i$ internally passive elements with respect to the ordering of the ground set.
 
From Theorem \ref{thm:main}, we know that for any matroid $M$ with cycle system $\cs$, the number of bases with $i$ internally passive elements is given by the number of coparking functions of degree $i$. It is then natural to ask if one can find an explicit bijection between coparking functions and bases so that a coparking function of degree $i$ is sent to a basis with $i$ internally passive elements (under some ordering of the ground set).

For the case of cographic matroids, such an `activity-preserving' bijection was indeed described by Cori and Le Borgne in \cite{CorLeB} by employing a breadth-first version of Dhar's burning algorithm. Here one must choose an ordering on the ground set (edges) that defines both the activity of a basis (spanning tree), as well as the bijection itself. Any such choice leads to an activity-preserving bijection between $G$-parking functions and spanning trees.

One can view our Algorithms \ref{alg:basis-to-coparking} and \ref{alg:coparking-to-basis} as generalizations of breadth-first Dhar's algorithm. However, one can check that not all choices of ordering on the ground set lead to an activity-preserving bijection. On the other hand, we have not found a case where no ordering works, and it is not clear at this point what properties are required. Hence we ask the following.

\begin{question}
Suppose $M$ is a matroid with cycle system $\cs$ and associated coparking functions~$P^*$. Can one find an ordering of the groundset for which the bijection between~$P^*$ and ${\mathcal B}(M)$ described in Section~\ref{subsect:basis-to-cpkg} preserves activity?
\end{question}
 
 \subsection{Circuit chip-firing}
In the classical setting of chip-firing on a graph $G$, the non-root vertices $V = \{v_1, \dots, v_n\}$ (which determine a cycle system for the cographic matroid) give rise to a chip-firing rule for $G$. In particular, the set of non-root vertices define the \emph{reduced Laplacian} $\tilde{L}(G)$, the $n \times n$ symmetric matrix whose diagonal entries $\tilde{L}(G)_{i,i}$ are given by $\deg(v_i)$ and whose off diagonal entries $\tilde{L}(G)_{i,j}$ are given by $-|\{e:\text{$e$ is an edge incident to $v_i$ and $v_j$}\}|$.

Now, given a configuration of chips $\vec{c} \in {\mathbb Z}^n$ on $V$, one can `fire' a vertex $v_i \in V$ if $c_i \geq \deg(v_i)$, in which case one obtains the new configuration
\[\vec{d} = \vec{c} - \tilde{L}(G)(\vec{e}_i),\]
\noindent
where $\vec{e}_i$ is the standard basis vector corresponding to $v_i$. In a similar way, one can fire a subset $S \subset V$.

This setup has been extended to a more general setting, where one replaces the matrix $\tilde{L}$ with other $n \times n$ distribution matrices $L$. Here we assume the diagonal entries of $L$ are positive, while the off diagonal entries are nonpositive. In \cite{GuzKli}, Guzm\'an and Klivans have shown that one obtains a good theory of chip-firing whenever the matrix $L$ has the \emph{avalanche finite} property, so that repeated firings (in the above sense) eventually stabilize, from any given starting configuration. Such matrices are known as $M$-matrices in the literature, and can be characterized in a number of ways.  In particular, if $L$ is an $M$-matrix, then each element of ${\mathbb Z}^n/\im L$ (the so-called \emph{critical group} of $L$) has a unique superstable representative, as well as a unique critical representative.

Now, for any matroid $M$ with cycle system $\cs = \{C_1, C_2, \dots, C_g\}$, one can define a matrix $L(\cs)$ in an analogous way to the reduced Laplacian, where  $L(\cs)_{i,i} = |C_i|$ and $L(\cs)_{i,j} = -|C_i \cap C_j|$ for $i\neq j$. A natural question to ask is when this matrix is an $M$-matrix, leading to a theory of chip-firing with good notions of superstable and critical configurations. As mentioned above, the only cycle system on the graphic matroid of $K_5$ arises by considering $K_5$ as a cone on $K_4$. For this cycle system $\cs$, the matrix $L(\cs)$ is not an $M$-matrix, as its inverse has negative entries. 

It turns out that even for planar graphs (which always admit cycle systems), certain choices of cycle system $\cs$ will define a matrix $L(\cs)$ that is not an $M$-matrix.

\begin{question}
 Suppose $M$ is a matroid that admits a cycle system. When does there exist a choice of cycle system $\cs$ on $M$ such that $L(\cs)$ is an $M$-matrix?
\end{question}

\subsection{Algebraic aspects}
Our last open questions involve an algebraic approach to the study of coparking functions. Recall that if $G$ is a connected graph with nonsink vertices $[n]$, and ${\mathbb K}$ is some fixed field, one can define a monomial ideal $M_G$ in the polynomial ring $S = {\mathbb K}[x_1, \dots, x_n]$ whose standard monomials are the $G$-parking functions. Here we identify a monomial with its exponent vector. It follows that the ${\mathbb K}$-dimension of the algebra ${\mathcal A}_G = S/M_G$ is given by the number of spanning trees of $G$. Even more, the Hilbert series of $S/M_G$ (which equals the Hilbert polynomial since $S/M_G$ is Artinian) coincides with the polynomial $T_G(1,y)$ (in reverse order).

 This is the approach taken by Postnikov and Shapiro in \cite{PosSha}, where they also show that a related `power algebra' ${\mathcal B}_G = S/P_G$ has a basis given by the same collection of monomials corresponding to $G$-parking functions. In particular, the Hilbert series of ${\mathcal A}_G$ and ${\mathcal B}_G$ coincide. 
 Here $P_G$ is an ideal generated by powers of certain linear forms. Such power algebras fit into a larger class of objects studied by Ardila and Postnikov in \cite{ArdilaPostnikov}, and are related to the theories of fat point ideals, Cox rings, and box splines. 
 
 In \cite{PosSha} the authors introduce the class of \emph{monotone monomial ideals} and \emph{$I$-deformations}, of which the ideals $M_G$ and $P_G$ are special cases. They show that the $I$-deformations always have Hilbert functions that dominate those of the original monotone monomial ideals. This is particularly interesting since the monomial ideal is almost never an initial ideal of the deformation, so that usual degeneration techniques from the theory of Gr\"obner basis cannot be applied.

Now, if $M$ is a matroid with cycle system $\cs = \{C_1, \dots, C_g\}$ and ${\mathbb K}$ is a field, one can define a monomial ideal in a similar way. For any subset $\sigma \subseteq [g]$ and $i \in I$, define 
\[d_\sigma(i) = |\cs_{\sigma}\cap C_{i}|,\]
\noindent
the number of elements that appear in the unique union $\cs_\sigma$, and which are also elements of $C_i$.

Fix a field ${\mathbb K}$ and define $M_\cs = \langle m_\sigma \rangle$ to be the monomial ideal in the polynomial ring $S = {\mathbb K}[x_1, \dots, x_g]$ generated by the monomials 
\[m_\sigma = \prod_{i \in \sigma} x_i^{d_\sigma(i)},\]
\noindent
where $\sigma$ ranges over all nonempty subsets of $[g]$. By construction, the (exponent vectors of the) standard monomials of $M_\cs$ are the coparking functions associated to $\cs$. 
It is also not hard to see that $M_\cs$ is an Artinian monotone monomial ideal, and hence the tools from \cite{PosSha} can be applied. In particular the Hilbert polynomial of $S/M_\cs$ is given by the Tutte evaluation $T_M(x,1)$ (in reverse order), and the Hilbert function of any $I$-deformation of $M_\cs$ dominates that of $M_\cs$.

\begin{question}
For a matroid $M$ with cycle system $\cs$, can one compute the the Hilbert polynomial of an $I$-deformation of $M_\cs$?  For which deformations does this function coincide with the Tutte evaluation $T_M(x,1)$ (in reverse order)?  
\end{question}

In \cite{PosSha}, the authors also consider free resolutions of \emph{order ideals}, a class of monomial ideals which further generalize monotone monomial ideals. They in particular describe a cellular resolution of $M_G$ that is minimal only in the case that $G = K_{n+1}$ is the complete graph. In \cite{DocSan} the authors describe a \emph{minimal} cellular resolution of $M_G$ for an arbitrary graph $G$. The resolution here is supported on a complex obtained from the graphic hyperplane arrangement associated to $G$.

\begin{question}
    Can one describe a minimal (cellular) resolution of the ideal $M_\cs$?
    \end{question}

\appendix
\section*{Appendix}\label{appendix}
\addcontentsline{toc}{section}{Appendix} 

Here we discuss circuit systems for graphs with at most eight vertices and give examples of matroids with circuit systems but no fundamental circuit system.  For this part, take all graphs to be connected.  For graphs with more than five vertices, our results depend on computer searches using Sage.
\medskip

\noindent \textbf{Graphs with at most eight vertices.}
Every graph with at most five vertices has a cycle system. Indeed, all of these are planar except for $K_5$, which is a cone. 
\medskip
     
      \noindent\textsc{six-vertex graphs}. There are only three six-vertex graphs without cycle systems: the complete bipartite graph $K_{3,3}$ and the two others pictured in Figure~\ref{fig: 6 vertices}.
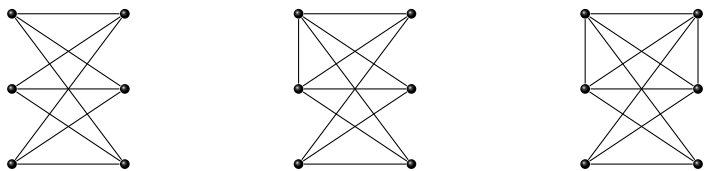
\begin{figure}[ht]
\centering
    \begin{tikzpicture}
      \begin{scope}
        \node[ball] (0) at (0,0) {};
        \node[ball] (1) at (0,1) {};
        \node[ball] (2) at (0,2) {};
        \node[ball] (3) at (1.5,0) {};
        \node[ball] (4) at (1.5,1) {};
        \node[ball] (5) at (1.5,2) {};
        \draw (0)--(3);
        \draw (0)--(4);
        \draw (0)--(5);
        \draw (1)--(3);
        \draw (1)--(4);
        \draw (1)--(5);
        \draw (2)--(3);
        \draw (2)--(4);
        \draw (2)--(5);
      \end{scope}
      \begin{scope}[xshift=1.5in]
        \node[ball] (0) at (0,0) {};
        \node[ball] (1) at (0,1) {};
        \node[ball] (2) at (0,2) {};
        \node[ball] (3) at (1.5,0) {};
        \node[ball] (4) at (1.5,1) {};
        \node[ball] (5) at (1.5,2) {};
        \draw (0)--(3);
        \draw (0)--(4);
        \draw (0)--(5);
        \draw (1)--(3);
        \draw (1)--(4);
        \draw (1)--(5);
        \draw (2)--(3);
        \draw (2)--(4);
        \draw (2)--(5);
        \draw (1)--(2);
      \end{scope}
      \begin{scope}[xshift=3.0in]
        \node[ball] (0) at (0,0) {};
        \node[ball] (1) at (0,1) {};
        \node[ball] (2) at (0,2) {};
        \node[ball] (3) at (1.5,0) {};
        \node[ball] (4) at (1.5,1) {};
        \node[ball] (5) at (1.5,2) {};
        \draw (0)--(3);
        \draw (0)--(4);
        \draw (0)--(5);
        \draw (1)--(3);
        \draw (1)--(4);
        \draw (1)--(5);
        \draw (2)--(3);
        \draw (2)--(4);
        \draw (2)--(5);
        \draw (1)--(2);
        \draw (4)--(5);
      \end{scope}
    \end{tikzpicture}
    \caption{The three 6-vertex graphs without cycle systems.}\label{fig: 6 vertices}
\end{figure}
    Beside these three graphs, there is only one other graph with six vertices that is neither planar nor coned.  It is pictured in Figure~\ref{fig: 6 np nc cs}.
    It has a circuit system since it is derived from $K_5$, which is a cone, by subdividing an edge (see~part~(\ref{item:example-subdivision}), below).
    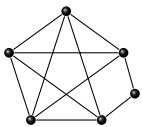
\begin{figure}[h]
    \centering
    \begin{tikzpicture}[scale=0.8]
      \def\r{1}
      \node[ball] (0) at (72*0+18:\r) {};
      \node[ball] (1) at (72*1+18:\r) {};
      \node[ball] (2) at (72*2+18:\r) {};
      \node[ball] (3) at (72*3+18:\r) {};
      \node[ball] (4) at (72*4+18:\r) {};
      \node[ball] (5) at (-18:1.2*\r) {};
      \draw (0)--(1)--(2)--(3)--(4);
      \draw (0)--(2)--(4)--(1)--(3);
      \draw (4)--(5)--(0);
      \draw (0)--(3);
    \end{tikzpicture}
    \caption{The unique non-planar, non-coned graph with six vertices with a circuit system.}\label{fig: 6 np nc cs}
    \end{figure}
    \medskip

    \noindent\textsc{seven-vertex graphs}.
    There are $853$ graphs with seven vertices.  Of these, we know of $779$ that have circuit systems and $61$ that have no circuit system.  The remaining $13$ are too large for our search methods (see~part~(\ref{item:cs-finding algorithm})).  For instance, one is pictured in Figure~\ref{fig: cs exists?}.  It has $223$ circuits and corank $10$, and thus has $\binom{223}{10}$, i.e., over $68$ quadrillion, potential circuit systems.
    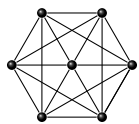
\begin{figure}[h]
    \centering
    \begin{tikzpicture}[scale=0.8]
        \foreach \i in {0,1,2,3,4,5} {
            \node[ball] (\i) at (\i*60:1) {};
        }
        \node[ball] (c) at (0,0) {};
        \draw (c)--(0);
        \draw (c)--(1);
        \draw (c)--(2);
        \draw (c)--(3);
        \draw (c)--(4);
        \draw (c)--(5);
        \draw (0)--(1);
        \draw (0)--(2);
        \draw (0)--(4);
        \draw (0)--(5);
        \draw (1)--(2);
        \draw (1)--(3);
        \draw (1)--(5);
        \draw (2)--(3);
        \draw (2)--(4);
        \draw (3)--(4);
        \draw (3)--(5);
        \draw (4)--(5);
        \draw (5)--(0);
    \end{tikzpicture}
    \caption{It is not known whether this graph has a circuit system.}\label{fig: cs exists?}
    \end{figure}
    
    We now try to account for the existence of circuit systems on some of these graphs. Recall that any graph that is planar, coned, or $K_{3,3}$-free has a circuit system. Of the $853$ seven-vertex graphs, there are~$722$ that are either planar or coned. There are then $17$ more that are neither planar nor coned but are $K_{3,3}$-free. Thus, three properties account for all but $40$ of the seven-vertex graphs with circuit systems.  
    
    Continuing, let $A$ be the set of connected seven-vertex graphs that are neither planar, coned, nor $K_{3,3}$-free.  It has~$114$ elements.  
    As discussed previously, if a graph is not biconnected, then it has a circuit system if and only if each of its biconnected components has a circuit system.  So circuit systems on these graphs are really questions about circuit systems on smaller graphs.  Among the $98$ biconnected elements of~$A$, 
    there are $27$ with fundamental circuit systems, $3$ with circuit systems but no fundamental circuit system (one of which appears in Figure~\ref{fig: 7 no fundamental}), $55$ with no circuit system, and $13$ whose status is unknown.  

    Non-isomorphic graphs may have isomorphic cycle matroids.  Seven-vertex graphs give rise to $533$ non-isomorphic matroids.  We have found that $462$  have circuit systems and $58$ do not.  There remain $13$ whose status is undetermined. 
    \medskip

    \noindent\textsc{eight-vertex graphs}.
    There are $11117$ connected eight-vertex graphs. Among these, there are exactly $6991$ with fundamental circuit systems.  We have found an additional $1185$ with circuit systems but no fundamental circuit system, and $433$ with no circuit system. The status of the remaining $2508$ is unknown.  There are $6916$ eight-vertex graphs that are either planar, coned, or $K_{3,3}$-free, and hence must have a circuit system.

    Let $A$ be the set of connected eight-vertex graphs that are neither planar, coned, nor $K_{3,3}$-free.  Then $A$ has $4201$ elements, $1225$ of which have fundamental circuit systems.  We know of $35$ more with circuit systems but no fundamental circuit system, and $433$ with no circuit system.  The status of the remaining $2508$ is unknown.  There are $3460$ elements of $A$ that are biconnected; $834$ have fundamental circuit systems, we know of an additional $25$ with a circuit system but no fundamental circuit system, $267$ with no circuit system, and the status of the remaining is unknown.

    The cycle matroids of the $11117$ eight-vertex graphs form $7303$ isomorphism classes.  We have found $4633$ that have circuit systems and $312$ that do not. There remain $2358$ whose status is undetermined.
\medskip

\noindent\textbf{Matroids with circuit systems but no fundamental circuit system.} 
 One can show that a matroid $M$ possesses a fundamental circuit system if and only if the dual matroid $M^*$ possesses a fundamental circuit system.  In particular, if $G$ is a graph with no circuit system, e.g., $K_{3,3}$, then the dual of its cycle matroid has no fundamental circuit system. However, the dual is cographic and, thus, has a circuit system.  
 
 A computer search shows that every graph with fewer than six vertices has a fundamental circuit system.  Figure~\ref{fig: 6 no fundamental} displays the three six-vertex graphs that have circuit systems (as do all planar graphs) but no fundamental circuit system (as verified using Sage).  The left-most is a ``double-cone'' over a four-cycle (rectangle), and the next two are obtained by successively deleting edges.
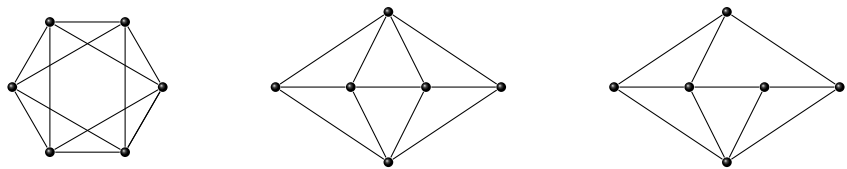
\begin{figure}[ht]
\centering
\begin{tikzpicture}
  \begin{scope}
    \foreach \i in {0,1,2,3,4,5} {
        \node[ball] (\i) at (\i*60:1) {};
    }
    \draw (0)--(1);
    \draw (0)--(2);
    \draw (0)--(4);
    \draw (0)--(5);
    \draw (1)--(2);
    \draw (1)--(3);
    \draw (1)--(5);
    \draw (2)--(3);
    \draw (2)--(4);
    \draw (3)--(4);
    \draw (3)--(5);
    \draw (4)--(5);
    \draw (5)--(0);
  \end{scope}
  \begin{scope}[xshift=2.5cm]
    \node[ball] (0) at (0,0) {};
    \node[ball] (1) at (1,0) {};
    \node[ball] (2) at (2,0) {};
    \node[ball] (3) at (3,0) {};
    \node[ball] (4) at (1.5,1) {};
    \node[ball] (5) at (1.5,-1) {};
    \draw (0)--(1);
    \draw (1)--(2);
    \draw (2)--(3);
    \draw (4)--(0);
    \draw (4)--(1);
    \draw (4)--(2);
    \draw (4)--(3);
    \draw (5)--(0);
    \draw (5)--(1);
    \draw (5)--(2);
    \draw (5)--(3);
  \end{scope}
  \begin{scope}[xshift=7cm]
    \node[ball] (0) at (0,0) {};
    \node[ball] (1) at (1,0) {};
    \node[ball] (2) at (2,0) {};
    \node[ball] (3) at (3,0) {};
    \node[ball] (4) at (1.5,1) {};
    \node[ball] (5) at (1.5,-1) {};
    \draw (0)--(1);
    \draw (1)--(2);
    \draw (2)--(3);
    \draw (4)--(0);
    \draw (4)--(1);
    \draw (4)--(3);
    \draw (5)--(0);
    \draw (5)--(1);
    \draw (5)--(2);
    \draw (5)--(3);
  \end{scope}
\end{tikzpicture}
\begin{minipage}{0.8\textwidth}
\caption{The three 6-vertex graphs with circuit systems but no fundamental circuit system.}\label{fig: 6 no fundamental}
\end{minipage}
\end{figure}
There are several graphs with seven vertices that have circuit systems but no fundamental circuit systems.  The one with the smallest number of edges is shown in Figure~\ref{fig: 7 no fundamental}.  A computer search shows it has exactly one circuit system, consisting of the edges of the square $3476$ and the triangles $145, 245,745, 134, 234, 567$.

\begin{figure}[h]
\centering
\begin{tikzpicture}
    \node[ball, label=right:$4$] (0) at (-18:1) {};
    \node[ball, label=right:$1$] (1) at (54:1) {};
    \node[ball, label=left:$2$] (2) at (126:1) {};
    \node[ball, label=left:$3$] (3) at (198:1) {};
    \node[ball, label=below:$5$] (4) at (270:1) {};
    
    \node[ball, label=left:$6$] (5) at (-0.9511,-1.9022) {};
    \node[ball, label=right:$7$] (6) at (0.9511,-1.9022) {};
    
    \draw (0)--(1);
    \draw (0)--(2);
    \draw (0)--(3);
    \draw (0)--(4);
    \draw (0)--(6);
    \draw (1)--(3);
    \draw (1)--(4);
    \draw (2)--(3);
    \draw (2)--(4);
    \draw (3)--(5);
    \draw (4)--(5);
    \draw (4)--(6);
    \draw (5)--(6);
\end{tikzpicture}
\caption{A non-planar graph with circuit system but no fundamental circuit system.}\label{fig: 7 no fundamental}
\end{figure}
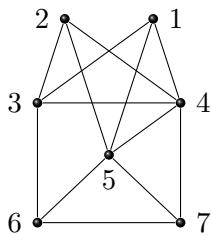
\bibliographystyle{plain}
\bibliography{cycle_systems}
\end{document}